\documentclass{ip-journal}

\usepackage{float}

\usepackage[backend=biber, style=numeric-comp, isbn=false, doi=false, url=false, maxbibnames=4]{biblatex}
\usepackage{mathtools}
\usepackage{import}
\usepackage{amsmath}
\usepackage{amsthm, thmtools, thm-restate}
\usepackage{amssymb}
\usepackage{graphicx}
\usepackage{enumitem}
\usepackage{listings}
\usepackage[hidelinks]{hyperref}
\usepackage{tikz}
\usepackage{tikz-cd}
\usetikzlibrary{calc,angles}
\usetikzlibrary{cd}
\usetikzlibrary{arrows.meta}

\DeclareFieldFormat*{title}{\mkbibemph{#1}}
\DeclareFieldFormat*{journaltitle}{#1}
\DeclareFieldFormat*{booktitle}{#1}

\renewbibmacro{in:}{%
  \ifentrytype{article}
    {}
    {\printtext{\bibstring{in}\intitlepunct}}}

\DeclareFieldFormat[article,periodical]{volume}{\mkbibbold{#1}}
\DeclareFieldFormat[article,periodical]{number}{\bibstring{number}~#1}
\renewbibmacro*{journal+issuetitle}{%
  \usebibmacro{journal}%
  \setunit*{\addspace}%
  \iffieldundef{series}
    {}
    {\newunit
     \printfield{series}%
     \setunit{\addspace}}%
  \printfield{volume}%
  \setunit{\addspace}%
  \usebibmacro{issue+date}%
  \setunit{\addcolon\space}%
  \usebibmacro{issue}%
  \setunit{\addcomma\space}%
  \printfield{number}%
  \setunit{\addcomma\space}%
  \printfield{eid}%
  \newunit}

\DeclareFieldFormat[article,periodical]{pages}{#1}

\DeclareFieldFormat{mrnumber}{%
  MR\addcolon\space
  \ifhyperref
    {\href{http://www.ams.org/mathscinet-getitem?mr=#1}{\nolinkurl{#1}}}
    {\nolinkurl{#1}}}

\renewbibmacro*{doi+eprint+url}{%
  \iftoggle{bbx:doi}
    {\printfield{doi}}
    {}%
  \newunit\newblock
  \printfield{mrnumber}%
  \newunit\newblock
  \iftoggle{bbx:eprint}
    {\usebibmacro{eprint}}
    {}%
  \newunit\newblock
  \iftoggle{bbx:url}
    {\usebibmacro{url+urldate}}
    {}}

\newcommand{\CR}{\mathcal R}

\newcommand{\F}{\mathcal F}
\newcommand{\A}{\mathcal A}
\newcommand{\M}{\mathcal M}
\newcommand{\Nc}{\mathcal N}
\newcommand{\N}{\mathbb N}

\newcommand{\VV}{\mathcal V}

\newcommand{\Pro}{\mathbb P}

\newcommand{\Z}{\mathbb Z}

\newcommand{\C}{\mathbb C}

\newcommand{\h}{\mathfrak h}

\newcommand{\X}{\mathcal{X}}
\newcommand{\Sc}{\mathcal{S}}

\newcommand{\CC}{\mathcal C}

\newcommand{\Msigma}{\Mod(S_{2g-1},\sigma)}
\newcommand{\Mbeta}{\Mod(S_g,[\beta])}

\def \Hs #1#2{H_#1(S_{#2},\Z)}

\DeclareMathOperator{\Ker}{Ker}

\DeclareMathOperator{\Sp}{Sp}
\DeclareMathOperator{\Fix}{Fix}

\DeclareMathOperator{\Jac}{Jac}

\DeclareMathOperator{\Homeo}{Homeo}

\DeclareMathOperator{\Diff}{Diff}

\DeclareMathOperator{\Ab}{Ab}

\DeclareMathOperator{\Mod}{Mod}

\DeclareMathOperator{\teich}{Teich}

\DeclareMathOperator{\PGL}{PGL}

\DeclareMathOperator{\SL}{SL}
\DeclareMathOperator{\SP}{Sp}

\DeclareMathOperator{\orb}{orb}
\DeclareMathOperator{\GL}{GL}
\DeclareMathOperator{\Id}{Id}
\DeclareMathOperator{\Prym}{Prym}
\DeclareMathOperator{\PSp}{PSp}
\DeclareMathOperator{\Ima}{Im}
\DeclareMathOperator{\Stab}{Stab}

\DeclareMathOperator{\Mor}{Mor}
\usepackage{
nameref,
hyperref,
}
\usepackage[capitalize]{cleveref}

\hypersetup{
    colorlinks=true,
    linkcolor=blue,
    filecolor=magenta,
    urlcolor=cyan,
    citecolor={purple!90!black}
    }

\declaretheorem[name = Theorem, refname = {theorem, theorems}, Refname = {Theorem, Theorems}, numberwithin=section]{theo}

\declaretheorem[name = Lemma, refname = {lemma, lemmas}, Refname = {Lemma, Lemmas}, numberwithin=section]{lemma}

\declaretheorem[name = Corollary, refname = {corollary, corollaries}, Refname = {Corollary, Corollaries},
sibling=theo]{cor}

\declaretheorem[name = Question, refname = {Question, Questions}, Refname = {Question, Questions}, numberwithin=section]{question}

\DeclarePairedDelimiter\dprod{\langle}{\rangle}

\usepackage{xcolor}

\theoremstyle{definition}

\newtheorem{defn}{Definition}[section]
\newtheorem{remark}{Remark}[section]

\usepackage{url}

\title{On the uniqueness of the Prym map}
\author{Carlos A. Serv\'an}
\address{Department of Mathematics\\}
\address{University of Chicago\\}
\email{\mbox{}cmarceloservan@uchicago.edu}

\begin{document}

\maketitle
\begin{abstract}
  The classical Prym construction associates to a smooth, genus $g$ complex curve $X$ equipped with
  a nonzero cohomology class $\theta \in H^1(X,\Z/2\Z)$, a principally polarized abelian variety (PPAV)
  $\Prym(X,\theta)$. Denote the moduli space of pairs $(X,\theta)$ by $\CR_g$, and let $\A_h$ be the
  moduli space of PPAVs of dimension $h$. The Prym construction globalizes to a holomorphic map of complex orbifolds
  $\Prym: \CR_g \to \A_{g-1}$. For $g\geq 4$ and $h \leq g-1$, we show that $\Prym$ is the unique nonconstant holomorphic map of complex orbifolds $F:\CR_g \to \A_h$.
  This solves a conjecture of Farb.
  A main component in our proof is a classification of homomorphisms $\pi_1^{\orb}(\CR_g) \to \SP(2h,\Z)$ for $h \leq g-1$.
  This is achieved using arguments from geometric group theory and low-dimensional topology.
\end{abstract}
\section{Introduction}
Let $X$ be a compact, connected, smooth, genus $g$ complex curve, and let $\Omega^1(X)$ be the space of holomorphic $1$-forms on $X$.
The \emph{Jacobian} of $X$,
\[ \Jac(X) := \frac{\Omega^1(X)^\vee}{H_1(X,\Z)}\]
is a $g$-dimensional principally polarized
abelian variety~(PPAV) canonically associated to $X$.

Let $\M_g$ be the moduli space of complex smooth genus $g$ curves, and let $\A_g$ be the moduli
space of PPAVs of dimension $g$. The Jacobian induces a holomorphic map, the \emph{Torelli map}
\[ J:\M_g \to \A_g \ \ , \ \ X \mapsto \Jac(X).\]

In a recent paper~\cite{farb2021global}, Farb showed that if $g \geq 3$ and $h \leq g$ then $J$ is the
\emph{unique} nonconstant holomorphic map of complex orbifolds $\M_g \to \A_h$.
In particular, extra data needs to be attached to smooth curves of genus $g$ in order to
associate, in a way that respects orbifold structures, a PPAV of dimension less than $g$ to each such curve. An example
of such a construction has been known to exist since over 100 years~\cite{farkas2011prym}, as we now explain.

 \medskip

 \noindent\textbf{The Prym construction.} Prym varieties~\cite{farkas2011prym}, named as such by Mumford in honor of Friedrich Prym
 (1841-1915), provide
a classical example of a way to obtain PPAVs of dimension $g-1$ from smooth curves of genus $g$.
Any nonzero $\theta \in H^1(X,\Z/2\Z)$ defines an unbranched double cover
\[ p:Y \to X, \]
with deck transform $\sigma$, and where $Y$ is a curve of genus $2g-1$. The map $p$ induces a map $\Jac(p):\Jac(Y) \to \Jac(X)$, between the jacobians of both curves.

The \emph{Prym variety} associated to $(X,\theta)$ is defined as (for more details and a explicit description see \Cref{subsubsec:prym_map})
\begin{equation} \Prym(X,\theta) := \ker \Jac(p) \in \A_{g-1} .\end{equation}

\medskip

\noindent\textbf{Moduli space of Prym varieties.} The Prym construction globalizes as follows.
Let
\begin{equation*}
  \begin{aligned}
    \CR_g := \{(X,\theta_X) : X &\mbox{ smooth complex curve of genus $g$,} \\
    &\mbox{ and } \theta_X \in H^1(X,\Z/2\Z)^* \}/\sim
  \end{aligned}
  \end{equation*}
be the space
of equivalence classes of pairs $(X,\theta_X)$, where $(X_1,\theta_1) \sim (X_2,\theta_2)$ if there exists a
biholomorphism $f:X_1 \to X_2$, with $f^*(\theta_2) = \theta_1$. As we explain in more detail
below in this introduction, $\CR_g$ is a complex orbifold
and the Prym construction globalizes to a map of complex orbifolds
\[ \Prym: \CR_g \to \A_{g-1} \ \ , \ \ (X,\theta_X) \mapsto \Prym(X,\theta_X).\]

Our main result shows that, as conjectured by Farb in~\cite{farb2021global}, $\Prym$ is \emph{holomorphically rigid}.

\begin{restatable}[\bf{Rigidity of $\Prym$}]{theo}{prymrigid}\label{theo:prym_rigid}
    Let $g \geq 4$ and let $h \leq g-1$. Let $F:\CR_g \to \A_h$ be a nonconstant holomorphic map of complex orbifolds\footnote{See \Cref{defn:orbi}}.
   Then $h = g-1$ and $F = \Prym$.
 \end{restatable}

 The proof of \Cref{theo:prym_rigid} uses in a fundamental way that $g \geq 4$. I do not know if the statement holds true also for $g = 2,3$.

 \medskip
 \noindent\textbf{Two orbifold structures on $\CR_g$.} There exist two natural orbifold structures on $\CR_g$, which
 give very different results with respect to maps to $\A_h$ (see \Cref{theo:triviality_rg}). Here we provide a brief description of the two orbifold structures and refer to \Cref{sec:orbi_str} for
 the details.

 Let $S_g$ be a closed surface of genus $g$, let $[\beta] \in
 H_1(S_g,\Z/2\Z)^*$, and let $p:S_{2g-1} \to S_g$
 be the
 associated double cover with deck transform $\sigma$. Let $\Mod(S_g)$ be the mapping class group of $S_g$, and define
 \[ \Mod(S_g,[\beta]):= \Stab_{\Mod(S_g)}([\beta]) \]
 as the stabilizer of $[\beta]$, with respect to the action of $\Mod(S_g)$ on $H_1(S_g,\Z/2\Z)$. Similarly, define
 \[ \Msigma := C_{\Mod(S_{2g-1})}([\sigma]) \]
 as the centralizer of $[\sigma]$.

Both $\Mbeta$ and $\Msigma$ act on Teichm\"uller space $\teich(S_g)$, and the two orbifold structures on $\CR_g$ come from considering
$\Mbeta$ or $\Msigma$ as the orbifold fundamental group of $\CR_g$. If we consider $\pi_1^{\orb}(\CR_g) = \Msigma$, then \emph{every} point of
$\CR_g$ is an orbifold point of order at least $2$. This phenomenon is
  akin to both $\Sp(2g,\Z)$ and $\PSp(2g,\Z)$ acting on Siegel upper half-space $\h_g$ and giving the same quotient $\A_g$, but
  different orbifold structures on $\A_g$. Unless otherwise specified,
  we always consider $\A_g$ with the orbifold structure given
  by $\Sp(2g,\Z)$.

  The difference
  between  these two orbifold structures on $\CR_g$ seems to be elided in the literature, yet as the following
  theorem shows, the inclusion of
  the involution $\sigma$ is fundamental to our
  results. Let $\hat{\CR}_g$ denote the orbifold structure on $\CR_g$ with $\pi_1^{\orb}(\hat{\CR}_g) = \Mbeta$.

  \begin{theo}\label{theo:triviality_rg}
    Fix $g \geq 4$ and $h \leq g-1$. Then, any holomorphic map $F: \hat{\CR}_g \to \A_h$ of complex orbifolds
    is constant.
  \end{theo}
  \begin{remark} If one considers only effective group actions in the definition of orbifolds (\Cref{defn:orbi}),
    then \Cref{theo:triviality_rg} is not
    correct. The action of $\Msigma$ on $\teich(S_g)$ factors through $\Mbeta$, and similarly the action of $\Sp(2h,\Z)$
    on $\h_h$ factors
    through $\PSp(2h,\Z)$. Hence, for effective actions there is no obstruction at the level of homomorphisms $\Mbeta
    \to \PSp(2h,\Z)$, and the Prym construction globalizes to a holomorphic map of complex orbifolds.  I do not know if the analogous statement to \Cref{theo:prym_rigid} holds in this setting but it
    will entail answering the following.
    \begin{question} Fix $g$ and $h\leq g-1$. Classify homomorphisms
      \[\phi: \Mbeta \to \PSp(2h,\Z).\]
      More generally, classify homormophisms $\phi:\Mbeta \to \PGL(m,\C)$ for any $m$.
    \end{question}

    \end{remark}

  \medskip

  \noindent \textbf{Prym representation.} In the same way as the standard symplectic representation of $\Mod(S_g)$ is associated to the Torelli
  map, the Prym map has an associated representation
  \[ \Prym_*: \Msigma \to \Sp(2g-2),\Z). \]
  The first step in the proof of \Cref{theo:prym_rigid} is the following purely group theoretic result.
  It shows that $\Prym_*$ exhibits a similar level of rigidity as that of the
standard symplectic representation for $\Mod(S_g)$.

  \begin{theo}[\bf{Rigidity of $\Prym_*$}]\label{theo:rigidity_group}Let $g \geq 4$ and $m \leq 2(g-1)$. Let $\phi:
  \Mod(S_{2g-1},\sigma) \to \GL(m,\C)$. The following holds,
  \begin{enumerate}
  \item If $m < 2(g-1)$ then $\Ima(\phi)$ is cyclic of order at most 4.
  \item Let $\chi:\Msigma \to \C^*$ be a group homomorphism. If $m = 2g-2$ then $\Ima(\phi)$ is either cyclic of order at
    most 4 or $\phi$ is
    conjugate to the map:
    \[ f \to \chi(f) \Prym_*(f),\]
    where $\chi(f)^4 =1$.
  \end{enumerate}
\end{theo}
Note that, unlike the case of $\Mod(S_g)$ (as shown in~\cite{Franks-Handel,korkmaz}), the group $\Msigma$ has a \emph{nontrivial} (infinite-image) linear representation of
dimension less than $2g$, and the same is true for $\Mbeta$ when $g$ is even (see \Cref{theo:rigidity_group_d}).

Since $\A_h$ for $h \geq 1$ is a $K(\pi,1)$ in the category of orbifolds, \Cref{theo:rigidity_group} implies
the following (see also \Cref{cor:symp_rep}).
\begin{cor}\label{cor:continuous_rigidity} Fix $g \geq 4$ and $h \leq g-1$. Let $F:\CR_g \to \A_h$ be a continuous map of orbifolds. If $h = g-1$, then
  $F$ is homotopic to $\Prym$. Otherwise, there exists a cyclic cover $\tilde{\CR}_g$ of $\CR_g$ of order at most $4$, so that
  the induced map $\tilde{F}:\tilde{\CR}_g \to \A_h$ is homotopic to a constant map.
\end{cor}

\noindent\textbf{Strategy of proof of\Cref{theo:prym_rigid,theo:triviality_rg}.} Our proof follows the general strategy laid out by Farb in \cite{farb2021global}, see \Cref{sec:farb_proof}. The two main aspects of
the proof  are the topological and
holomorphic sides of the story.
\begin{enumerate}
\item In \Cref{sec:Topology}, we classify low-dimensional linear and symplectic representations of $\Mbeta$
  and $\Msigma$. Our approach is based on (and extends) the results of Franks-Handel, and Korkmaz \cite{korkmaz,Franks-Handel},
  which classify linear representations for the full mapping class group $\Mod(S_g)$. A key ingredient in our proof is
  to prove connectedness of the complex of curves $\Nc_1(S_g)$ (see \Cref{subsec:complex-curves}). This covers the first step in Farb's proof.

\item In \Cref{sec:holo}, we add the assumption of holomorphicity for the map $F:\CR_g \to \A_h$ to deduce \Cref{theo:triviality_rg} and
  reduce the proof of \Cref{theo:prym_rigid} to the case of $h = g-1$ and
  $F$ homotopic to $\Prym$. In order to avoid orbifold issues when dealing with the $h =g-1$ case,
  we will pass to a suitable (smooth) cover $\CR_g[\psi]$ of $\CR_g$. Steps 2-4 in Farb's
  proof~\cite[Section 1]{farb2021global} for the rigidity of the Torelli map $J: \M_g \to \A_g$, extend to our
  case without modifications. Step 5, the existence of $\A_{g-1}$-rigid curves, requires some minor
  modifications. They arise due to our use of finite non-Galois covers of $\overline{\M_g}$. An
  alternative approach using variations of hodge structures is also given.
\end{enumerate}

\noindent\textbf{Speculation.} One would be tempted to conjecture that the only maps $\CR_g \to \A_g$ are
given by $\CR_g \xrightarrow{\Prym} \A_{g-1} \to \A_g$ and $\CR_g \to \M_g \xrightarrow{J} \A_g$.
For our proof strategy to work, one would first need to classify homomorphisms \[ \Msigma \to \Sp(2g,\Z).\]
A step in this direction is given by the recent work of Kasahara~\cite[Cor 1.3]{kasahara2023crossed}, which shows
that there are no irreducible representations $\Mod(S_g) \to \GL_{2g+1}(\C)$ for $g$ big enough.
\mbox{}\\

\noindent\textbf{Acknowledgments.} I am very grateful to my advisor Benson Farb for suggesting the problem, his guidance
and constant encouragement throughout the whole project, and for numerous comments on earlier drafts of the paper. I
would like to thank Curtis McMullen and Dan Margalit for comments on an earlier draft;
Eduard Looijenga for explaining to me properties of $\partial \M_g$ and VHS; and Frederick Benirschke and Casimir Kothari for
many insightful conversations. I would also like to thank the anonymous referees for their many suggestions
to help with the readability and clarity of the paper.

\section{Orbifold structures on $\CR_g$}\label{sec:orbi_str}
In this section we show how to give $\CR_g$ the structure
of a complex orbifold. First, let us briefly recall the definition of orbifold and maps between orbifolds~\cite[Remark
2.1]{farb2021global}.

\medskip
\begin{defn}[\textbf{Orbifolds and maps between orbifolds}]\label{defn:orbi}
  Let $X$ be a simply connected manifold (resp. complex manifold) and let $\Gamma$ be a group acting properly discontinuously on $X$ by
  homeomorphisms (resp. biholomorphisms), but not necessarily
  freely nor effectively. Then the quotient $X/\Gamma$ is a topological (resp. complex) \emph{orbifold}. Define
  $\pi_1^{\orb}(X/\Gamma) := \Gamma$ as the \emph{orbifold fundamental group} of $X/\Gamma$. Let $Y/\Lambda$ be another orbifold, and $\rho:\Gamma \to \Lambda$ a
  group homomorphism.  A continuous (resp. holomorphic) map in the category of orbifolds $F:X/\Gamma \to Y/\Lambda$ is a
  map so that there exists a
  continuous (resp. holomorphic) lift $\tilde{F}: X \to Y$ that \emph{intertwines} $\rho$:
  \[ \tilde{F}(\gamma.x) = \rho(\gamma).\tilde{F}(x) \ \ \ \mbox{ for all $x \in X, \gamma \in \Gamma$.}\]
  If this is the case we denote $\rho$ by $F_*:\Gamma \to \Lambda$. Note that postcomposition of $F_*$ with an inner automorphism $c_\ell$
  of
  $\Lambda$ changes $\tilde{F} \to \ell \circ \tilde{F}$, so that $F_*$ is defined up to postcomposition with inner
  automorphisms of $\Lambda$.
\end{defn}

\begin{remark}
  If $\Gamma$ acts effectively, our definition agrees with Thurston's definition of \emph{good}
  orbifold~\cite[Ch.13]{thurston1979geometry}, with $X$ being the orbifold universal cover of $X/\Gamma$.
  For noneffective group actions our defintion can be more restrictive, but captures the nature of
  examples such as the Torelli map and the $\Prym$ map.
\end{remark}

\noindent Let $S_g$ be a closed surface of genus $g$. The mapping class group $\Mod(S_g)$ is defined as
\[ \Mod(S_g) := \pi_0(\Diff^+(S_g)). \]
Let $\teich(S_g)$ denote the \emph{Teichm\"uller} space of $S_g$, the space of holomorphic structures
on $S_g$ up to isotopy. $\Mod(S_g)$ acts on $\teich(S_g)$ properly discontinously, but not freely, by
biholomorphisms. Let $[\beta] \not= 0 \in H_1(S_g,\Z/2\Z)$ and define
\[ \Mbeta := \Stab_{\Mod(S_g)}([\beta]), \]
as the stabilizer of $[\beta]$ in $\Mod(S_g)$. Then define
\[ \hat{\CR}_g := \teich(S_g)/\Mbeta. \]
In particular, $\hat{\CR}_g$ has the structure of a complex orbifold with $\pi_1^{\orb}(\hat{\CR}_g) = \Mbeta$.
Furthermore, $\hat{\CR}_g$ is in bijective correspondence with $\CR_g$ and thus endows $\CR_g$ with an orbifold structure.
The forgetful map
\[ \hat{\CR}_g \to  \M_g,\]
is a finite orbifold cover of $\M_g$ of degree $2^{2g}-1$, each fiber over a generic $X \in \M_g$ corresponding to $H_1(X,\Z/2Z)^*$.

One of the goals of this paper is to classify all holomorphic maps of complex orbifolds $\hat{\CR}_g \to \A_h$
for $h\leq g-1$. Define the map,
\[ \widehat{\Prym}:\hat{\CR}_g \to \A_{g-1} \ \ , \ \ (X,\theta) \to \Prym(X,\theta). \]
\Cref{theo:triviality_rg} shows that $\widehat{\Prym}$ \emph{cannot} be a map in the category
of complex orbifolds.

\medskip
  \noindent\textbf{Obstruction.} The obstruction to realize $\widehat{\Prym}$ as map of orbifolds is the \emph{nonexistence}
  of nonfinite representations $\phi:\Mbeta \to \Sp(2g - 2,\Z)$. As we explain in more detail in \Cref{sec:Topology},
  the Prym construction defines a representation:
  \[ \widehat{\Prym}_*: \Mbeta \to \PSp(2g-2,\Z),\]
  which does not lift to a symplectic representation.
  Thus, there is an associated nonsplit central $\Z/2\Z$ extension:
  \[ 1 \to \Z/2\Z \to H \to \Mbeta \to 1 \]

  By definition, there is a representation $H \to \Sp(2g-2,\Z)$ and $H$ acts on $\teich(S_g)$ via $\Mbeta$ so that
  every point is an orbifold point of order at least 2.
  Thus, $\CR_g$ can be endowed with an orbifold structure for which the Prym construction does define
  a holomorphic map $\Prym$ in the category of complex orbifolds. In fact, there is a concrete description of $H$ and this
  alternative orbifold structure, as we now explain.

\medskip
  \noindent\textbf{Moduli space of double covers.}  Let $Y$ be a complex smooth genus $2g-1$ curve, and $\sigma_Y:Y \to Y$
  a fixed-point free biholomorphic involution. Say that two such pairs $(Y_1,\sigma_{Y_1})$ and
  $(Y_2,\sigma_2)$ are equivalent if there is a biholomorphism $f:Y_1 \to Y_2$
  such that $f^{-1}\sigma_2 f = \sigma_1$. Then, there is a bijection
  \[ \phi: \{ [(Y,\sigma_Y)]\} \to \CR_g \ \ , \ \  [(Y,\sigma_Y)] \to [(Y/\sigma_Y,\theta_Y)]\]
  where $\theta_Y$ is given by the monodromy of the covering $p:Y \to Y/\sigma_Y$.


  Let $\sigma$ be a fixed-point free involution on the closed surface $S_{2g-1}$, and let $[\sigma]$
  be its class in $\Mod(S_{2g-1})$. Let $\Fix([\sigma]):= \teich(S_{2g-1})^{[\sigma]}$
  and define
  \[ \Msigma:= C_{\Mod(S_{2g-1})}([\sigma]).\]
  Then, there is an exact sequence
  \[ 1 \to \dprod{\sigma} \to \Mod(S_{2g-1},\sigma) \to \Mod(S_g, [\beta])\to 1 ,\]
  and, via the bijection $\phi$,
  \[ \CR_g = \Fix([\sigma])/\Msigma \]
  so that $\pi_1^{\orb}(\CR_g) = \Msigma$.

  Furthermore, $\phi$ induces a $2:1$ map of complex orbifolds (but a biholomorphism in the complex category)

   \begin{center}
   \begin{tikzcd}
     \Fix([\sigma]) \ar[d] \ar[r, "\tilde{\phi}"] & \teich(S_g) \ar[d]\\
     \ar [d,"\F_1"] \CR_g \ar[r, "\phi"] & \hat{\CR}_g \ar[d, "\F_2"] \\
     \M_{2g-1} & \M_g
   \end{tikzcd}
   \end{center}

   The $\F_i$ are the respective forgetful maps, but only $\F_2$ gives an orbifold covering.

   Thus, viewing $\CR_g$ as equivalence classes of curves with an involution is precisely
   the alternative orbifold structure stated at the end of the previous section, and
   the Prym construction induces a holomorphic map of complex orbifolds,
   \[ \Prym:\CR_g \to \A_{g-1} \ \ , \ \ (Y,\sigma_Y) \to \Prym(Y/\sigma_Y,\theta_Y).\]
  The difference between $\CR_g$ and $\hat{\CR}_g$ is precisely the difference between having covers of $\M_g$
  given by $G$-structures or by $G$-covers (cf. \cite[Ch 16, p 525-526]{arbarello2011geometry}), in our case $G = \Z/2\Z$.


\section{Topological results}\label{sec:Topology}
Let $S_g$ be a closed surface of genus $g \geq 1$, and $[\beta] \in H_1(S,\Z/2\Z)^*$. Then, there is a (unique up
to isomorphism) double cover \[ p:S_{2g-1} \to S_g, \]
with deck transform $\sigma$, and monodromy given by intersection with $[\beta]$.
By the work of Birman-Hilden~\cite{Birman-Hilden},
\[ \Msigma = \pi_0(\Diff^+(S_{2g-1},\sigma)), \]
which gives an exact sequence,
\[ 1 \to \dprod{\sigma} \to \Msigma \to \Mbeta \to 1. \]
\begin{remark}
 Note that $\Mod(S_g)$ acts
 transitively on $H_1(S_g,\Z/2\Z)$, and so any choice of $[\beta] \in H_1(S_g,\Z/2\Z)^*$ gives
 conjugate subgroups $\Mod(S_g,[\beta])$ within $\Mod(S_g)$.
The same remark applies to $\Mod(S_{2g-1},\sigma)$ in $\Mod(S_{2g-1})$, for different choices of $\sigma$.
\end{remark}
\medskip
\noindent \textbf{Prym representation.} For any $f \in \Msigma$, denote by $f_*$ its induced action on $H_1(S_{2g-1},\Z)$. As $f\sigma = \sigma f$, $f_*$
preserves the eigenspaces of $\sigma_*$. In particular, $f_*$
preserves $H_1(S_{2g-1},\Z)^-$, which consists of $\sigma$-anti-invariant elements.

Let $\hat{i}_{-} := \frac{1}{2} \hat{i}$, for $\hat{i}$ the restriction of the intersection pairing on $H_1(S_{2g-1},\Z)$
to $\Hs{1}{2g-1}^-$. Then $f_*$ will further preserve $\hat{i}_{-}$; thus by choosing a symplectic basis we obtain a representation
\[
  \Prym_*: \Msigma \to \Sp(2g - 2,\Z),
\]
called the \emph{Prym representation} of $\Msigma$.

Let $f \in \Mbeta$, then there is a lift $\tilde{f} \in \Msigma$, well-defined up to composition
with $\sigma$. As $\sigma_*$ acts as $-1$ on $H_1(S_{2g-1},\Z)^-$, the Prym representation induces a \emph{projective Prym representation},
\[ \widehat{\Prym}_*: \Mbeta \to \PSp(2g - 2,\Z).\]

In this section we build on the results of Franks-Handel and Korkmaz\cite{korkmaz,Franks-Handel}, to classify low-dimensional linear and symplectic representations
of $\Mbeta$ and $\Msigma$.

\begin{remark} The existence of $\chi:\Msigma \to \C^*$ in \Cref{theo:rigidity_group} is possible due to the fact that
  \[ \Msigma^{\Ab} \cong \Z/4\Z.\]
  Similarly,
  \[\Mbeta^{\Ab} \cong \Z/d\Z\]
  where $d = 2$ for $g$ even and $4$ otherwise~(see Sato~\cite[Theorem 1.2]{Sato}, or the appendix for an alternative proof of the even case).
\end{remark}

A similar rigidity result as of \Cref{theo:rigidity_group} holds for $\Mod(S_g,[\beta])$,
\begin{theo}\label{theo:rigidity_group_d} Let $g \geq 4$ and $m \leq 2g - 2$. Let $\phi:\Mod(S_g,[\beta]) \to
  \GL(m,\C)$. Then the following holds,
  \begin{enumerate}
  \item If either,
    \begin{enumerate}
    \item $m < 2g - 2$ or,
    \item $m = 2g - 2$ and $g$ odd or,
    \item $m = 2g - 2$, $g$ even and $\Ima(\phi) \subset \SL(m,\C)$.
    \end{enumerate}
    Then, $\Ima(\phi)$ is abelian, so it is a quotient of $\Z/4\Z$.
  \item Otherwise, $\phi$ is induced from a representation $\tilde{\phi}:\Msigma \to \GL(m,\C)$ such that $\tilde{\phi}(\sigma) = 1$.
    In particular, $\phi(T_a^2) = \pm i\Id$, for \\$\hat{i}_2([a],[\beta]) = 1$.
  \end{enumerate}
\end{theo}
In particular, this shows that $\widehat{\Prym}_*$ does \emph{not} lift to a linear representation.
In fact, let
\[ 1 \to \Z/2\Z \to H \to \Mbeta \to 1 \]
be the central extension determined by $\widehat{\Prym}_*:\Mbeta \to \PSp(2g - 2,\Z)$. Then
\[ \Msigma \cong H,\]
where the isomorphism is given by $\tilde{f} \to (f,\tilde{f}_*)$. Thus,
\begin{cor} The sequence
  \[ 1 \to \dprod{\sigma} \to \Msigma \to \Mbeta \to 1, \]
  does not split.
\end{cor}



\textbf{Proof outline for\Cref{theo:rigidity_group,theo:rigidity_group_d}.} Here we briefly sketch the main ideas
used in the proofs of \Cref{theo:rigidity_group,theo:rigidity_group_d}, the details will be given in the subsequent sections.
First observe that any $[\beta] \in H_1(S_g,\Z/2\Z)^*$ can be represented by a nonseparating simple closed
curve $b$. Then, there exists a subsurface (with boundary) $R \subset S_g$ of genus $g-1$ so that
\[ \Mod(R) \subset \Mbeta.\]
Results of Franks-Handel and Korkmaz applied to $\Mod(R)$, then give constraints on the restriction of $\phi$ to
$\Mod(R)$.

Moreover, as any element in $\Mod(R)$ fixes a point of $S_g$, $\Mod(R)$ lifts to $\widetilde{\Mod(R)} \subset \Msigma$
and one can check that $\Prym|_{\widetilde{\Mod(R)}}$ is precisely the symplectic representation of $\Mod(R)$.

In order to extend our knowledge of $\phi$ to the whole of $\Mbeta$, we find good generating sets for
$\Mbeta$. This is accomplished in two ways:
\begin{enumerate}
  \item A key property of $\Mod(S_g)$ is the fact that all nonseparating Dehn twists $T_a$ are conjugate to each other. This is no longer true in $\Mbeta$ and
    the results in \Cref{subsec:conj} give a classification of (powers of) such Dehn twists in $\Mbeta$ up to conjugation.
    As a corollary, there exists a normal generating set for $\Mbeta$ composed of only three types of Dehn Twists.
  \item \Cref{subsec:complex-curves} describes properties of the action of $\Mbeta$ on a modified complex of curves $\Nc(S_g)$. We show that
    $\Nc(S_g)$ is connected and $\Mbeta$ acts transitively on the edges and vertices of $\Nc(S_g)$. Thus, via a geometric group theory argument, there exists
    an additional generating set for $\Mbeta$.
\end{enumerate}
By using these two distinct generating sets, we are then able to constrain all low-dimensional representations of
$\Mbeta$ (except for the last item of \Cref{theo:rigidity_group_d}).
In \Cref{subsec:rel-msigma}, we lift the results from \Cref{subsec:conj,subsec:complex-curves} to $\Msigma$ and are able
to conclude all but the second item of \Cref{theo:rigidity_group}. The final ingredient in the proof is an explicit
(finite) generating set for $\Mbeta$, found by Dey-Dhanwani-Patil-Rajeevsarathy in~\cite{dey2021generating}, on which one can check that $\phi$ has the desired form.


\subsection{Conjugation in $\Mbeta$}\label{subsec:conj}
Let $\hat{i}_2$ be the algebraic intersection pairing mod $2$ on $H_1(S_g,\Z/2\Z)$. In what follows all homology classes are mod $2$. Let $\Sc(S_g)$ denote the set of isotopy classes of nonseparating simple closed curves (SCCs, from now on) in $S_g$. Theb
action of $\Mbeta$
splits $\Sc(S_g)$ into three orbits, classified as follows (See also~\cite[Lemma 2.3.1]{Donagi} for the
analogous statement for homology classes).
\begin{lemma}
  Let $\alpha_1,\alpha_2$ be a pair of nonseparating SCCs
  in $S_g$. The following are necessary and sufficient conditions for there to be a
  $\phi \in \Homeo^+(S_g)$ such that $\phi(\alpha_1) = \alpha_2$ and $\phi_*([\beta]) = [\beta]$.
  \begin{enumerate}
  \item $\hat{i}_2([a_1],[\beta]) = \hat{i}_2([a_2],[\beta]) = 1$.
  \item $\hat{i}_2([a_1],[\beta]) = \hat{i}_2([a_2],[\beta]) = 0$, and either both
    $[a_i] \not=[\beta]$ or both $[a_i] = [\beta]$.
  \end{enumerate}
\end{lemma}
\begin{proof}
  Let $\alpha$ be a nonseparating simple closed curve in $S_g$. Observe that if $\hat{i}_2([\alpha],[\beta])
  = c$, there exists\footnote{Extend $\alpha$ to a geometric simplectic basis. Locally, there are only three choices for a
    representative of $[\beta]$ and they can be glued together as needed.} a simple closed curve $b$ representing $[\beta]$ and intersecting
  $\alpha$ transversely $c$ times. Let $\alpha_1$ and $\alpha_2$ be two simple closed curves in $S_g$ with
\[ \hat{i}_2([\alpha_1],[\beta])
  = \hat{i}_2([\alpha_2],[\beta]) = 1.\]
By the previous observation, there exist two $2$-chains
  $(\alpha_i,b_i)$ with $[b_i] = [\beta]$. Thus, by the change of coordinates principle~\cite[Ch 1, Sec 3]{FM}, there is a $\phi \in \Homeo^+(S_g)$ so that $\phi(\alpha_1) = \alpha_2$
  and $\phi(b_1)= b_2$. In particular, $\phi_*([\beta]) = [\beta]$ and the first claim follows.

  Now suppose that $\hat{i}_2([\alpha],[\beta]) =0$. If $[\alpha] = [\beta]$,
  the statement follows since $\Homeo^+(S_g)$ acts transitively on nonseparating simple closed curves and any $\phi$ with $\phi(\alpha) = \beta$
  fixes $[\beta]$. Suppose that $[\alpha] \not=[\beta]$.
  Let $b$ be a simple closed curve representing $[\beta]$ and \emph{not intersecting} $\alpha$. In particular, $\alpha$ is \emph{nonseparating}
  in $S_g-b$. Let $\alpha_1,\alpha_2$ be two simple closed curves such that
  \[ \hat{i}_2([\alpha],[\beta])
    = \hat{i}_2([\alpha_2],[\beta]) = 0.\]
  Then, there are two $b_i$ representing $[\beta]$ such that $\alpha_i \cap b_i = 0$. Let $\phi \in \Homeo^+(S_g)$ with
  $\phi(b_1) = b_2$. Then $\phi(\alpha_1)$ is nonseparating in the cut-surface $S_{b_2}$ obtained by cutting along
  $b_2$. Applying the change of coordinates again, there is a $\psi \in \Homeo^+(S_{b_2},b_2)$ such that
  $\psi(\phi(\alpha_1)) = \alpha_2$. Composing $\phi$ with the map $\overline{\psi} \in \Homeo^+(S_g)$, induced by
  $\psi$, the claim follows.
\end{proof}

\begin{cor}[\bf{Conjugation in $\Mbeta$}]\label{cor:conj_Mbeta}
  Let $a_1,a_2$ be a pair of isotopy classes of nonseparating SCCs
  in $S_g$. Let $T_{a_i}$ be the Dehn twists along $a_i$, then:
  \begin{enumerate}
  \item If $\hat{i}_2([a_1],[\beta]) = \hat{i}_2([a_2],[\beta]) = 1$ then $T_{a_1}^2$ and $T_{a_2}^2$ are
    conjugate in $\Mbeta$.
  \item If $\hat{i}_2([a_1],[\beta]) = \hat{i}_2([a_2],[\beta]) = 0$ and either for each $i$
    $[a_i] \not=[\beta]$ or for each $i$ $[a_i] = [\beta]$, then $T_{a_1}$ and $T_{a_2}$ are conjugate in
    $\Mbeta$.
  \end{enumerate}
\end{cor}
The importance of \Cref{cor:conj_Mbeta} lies on the following.

\begin{lemma}[\bf{Generating set-Twists}]\label{lemma:gen_dehn_pi1(Rg)}
  Let $g \geq 2$ and $[\beta]$ a nonzero class in $H_1(S_g,\Z/2\Z)$. $\Mbeta$ is generated by

  \begin{equation*}
    \begin{aligned}
      \{T_c^{\xi(c)}: c &\mbox{ nonseparating SCC in $S_g$ and } \xi(c) \in \{1,2\}, \\
      &\mbox{ with } \xi(c)= \hat{i}_2([c],[\beta]) + 1 \mod 2\}.
  \end{aligned}
  \end{equation*}
\end{lemma}
\begin{proof}
  Let $\Lambda_g[\beta]$ be the stabilizer of $[\beta]$ in $\SP(2g,\Z/2\Z)$, and consider the following exact sequence, given by reducing the symplectic representation
  mod $2$.
  \begin{center}
  \begin{tikzcd}
    1 \ar[r] & \Mod(S_g)[2] \ar[r] & \Mbeta \ar[r, "\Psi_2"] & \Lambda_g[\beta] \ar[r] &1,
  \end{tikzcd}
\end{center} $\Lambda_g[\beta]$
is generated by transvections of the form $\psi_2(T_{c_i})$ for $\hat{i}_2([c_i],[\beta]) =0$ (cf.\cite[Lemma 3.4]{Eduard1}).
Similarly $\Mod(S_g)[2]$ is generated by squares of nonseparating Dehn twists \cite[Thm 1]{Humphries1}, thus the claim follows.
\end{proof}


\subsection{Complex of curves}\label{subsec:complex-curves}

The generating set given by \Cref{lemma:gen_dehn_pi1(Rg)} is enough for providing bounds for the abelianization of
$\Mbeta$ (see appendix). Yet, in order to establish \Cref{theo:rigidity_group_d}, we need to make use of another generating set. For this purpose,
we examine the action of $\Mbeta$ on $\Sc(S_g)$.

As above, fix $[\beta] \in H_1(S_g,\Z/2\Z)^*$. Let $\Sc_1(S_g)$ be the set of isotopy classes $a$ of simple closed curves
on $S_g$ such that
$\hat{i}_2([a],[\beta]) = 1$, in particular $\Sc_1(S_g) \subset \Sc(S_g)$ as any separating curve is trivial in homology. For two isotopy classes $a,b$ of simple closed curves, let
\[ i(a,b) = \min \{ |\alpha \cap \beta|:  \alpha \in a, \ \beta \in b \} \]
denote their
\emph{geometric intersection number}~\cite[Ch 1.2.3]{FM}.
\begin{defn}[\bf{Complex of curves}]
  Let $\Nc_1(S_g)$ be the 1-complex with vertex set $\Sc_1(S_g)$. An edge $(a,c)$ between $a,c \in S_1(S_g)$ exists iff $i(a,c) = 1$ and $[a] + [c] \not= [\beta]$.
\end{defn}
The most important property of $\Nc_1(S_g)$ for our purposes, and whose proof will occupy most of this section, is the following.
\begin{lemma}\label{lemma:complex_con}
  For $g \geq 3$, $\Nc_1(S_g)$ is connected.
\end{lemma}

The proof of \Cref{lemma:complex_con} follows the same idea as when dealing with the standard complex of curves~\cite[Chapter 4]{FM}. We
  first define two associated $1$-complexes, the second of which contains $\Nc_1(S_g)$. We prove connectivity for each
  of them and then refine the paths to be in $\Nc_1(S_g)$.

  Define $\CC_1(S_g)$ to be the $1$-complex with the same vertex set as $\Nc_1(S_g)$ and edges
  between vertices $a,c$ if and only if $i(a,c) = 0$, i.e. if $a$ and $c$ admit \emph{disjoint} representatives.
  \begin{lemma}\label{lemma:disjoint_complex}
    $\CC_1(S_g)$ is connected for $g\geq 2$.
  \end{lemma}

  \begin{proof}
  Let $a, c \in \Sc_1(S_g)$. We proceed by induction on $i(a,c)$ , the case $i(a,c) = 0$ being clear. For $i(a,c) =1$, let
  $\alpha$ and $\gamma$ be representatives of $a,c$ in minimal position. It follows
  that $\alpha$ and $\gamma$ are part of a geometric symplectic basis $\nu$ for $H_1(S_g,\Z)$. Thus, there exist a multicurve
  representative of $[\beta]$ intersecting $\alpha$ and $\gamma$ only once. If $[a] + [c] = [\beta]$, then there is
  a curve $\delta$ with isotopy class $d$, with the following properties:
  \begin{enumerate}
  \item $\alpha \cap \delta = \emptyset$.
  \item $[d] + [c] \not=[\beta]$.
  \item $i(c,d) = 1$.
  \end{enumerate}
  Indeed, $\delta$ can be found by applying
  the change of coordinates principle~\cite[Ch 1.3]{FM}
  \footnote{The main idea is that, thanks to the classification of surfaces,
    up to conjugation with a diffeomorphism, we can change to a \emph{standard} picture, where it is easier to find representatives.}.
  Hence it is enough to assume that $[a] + [c] \not= [\beta]$. In this
  case there is a component of $[\beta]$ intersecting one of the other curves in the basis $\nu$, say $\gamma'$, once. The isotopy class of $\gamma'$ provides the path between $a$ and $c$ in $\CC_1(S_g)$.

  Now assume $i(a,c) \geq 2$, and let $\alpha,\gamma$ be as above. As before, there is a representative $\beta$ of $[\beta]$
  intersecting $\gamma$ only once and intersecting $\alpha$ transversely. Take two consecutive intersection points of
  $\gamma$ and $\alpha$. There are two cases,
  depending on the orientation at the intersections:

\begin{figure}[H]
  \begin{center}
    \def \svgwidth{1\textwidth}
    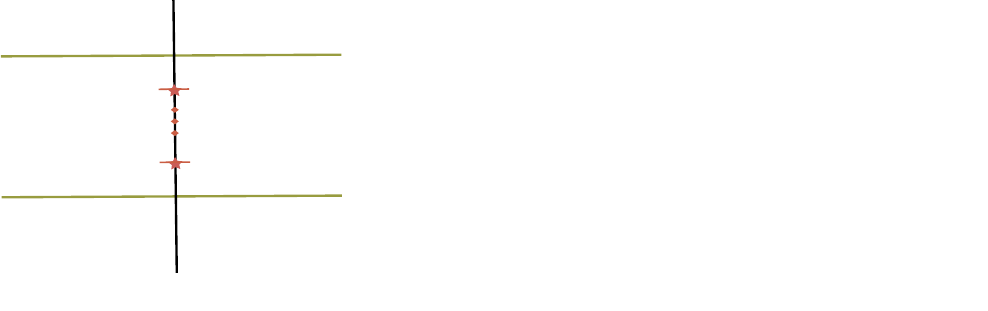
  \end{center}
  \caption{Path surgery.}\label{fig:path_surgery}
\end{figure}

  In either case, let $\gamma_1$ and $\gamma_2$ be SCCs constructed by the surgery described \Cref{fig:path_surgery}. As these curves travel parallel to
  $\gamma$ and the union gives all of $\gamma$ outside the neighboorhood of $\alpha$ depicted above, only one of
  the curves crosses $\beta$ along $\gamma$, say it is $\gamma_1$.
  Furthermore, $\gamma_1$ and $\gamma_2$ have a segment parallel to $\alpha$, and this segment will meet $\beta$
  in either an even or odd number of points. Depending on the parity, either $\gamma_1$ (even case) or
  $\gamma_2$ (odd case) will meet $\beta$ an odd number of times, and intersect both $\alpha$ and $\gamma$ in fewer than
  $i(a,c)$ points. By induction, there is a path between $a$ and $c$ and the claim follows.
\end{proof}

Next, define $\Nc\CC_1(S_g)$, to be the 1-complex with vertex set $\Sc_1(S_g)$ and where two classes
$a,c$ in $\Sc_1(S_g)$ are connected by an edge if and only if $i(a,c) =1$.
\begin{lemma}\label{lemma:1_complex}
  For $g\geq 2$, $\Nc\CC_1(S_g)$ is connected.
\end{lemma}

\begin{proof}

  Let $a, c \in \Sc_1(S_g)$, by \Cref{lemma:disjoint_complex}, we can assume $i(a,c) = 0$. Thus, it is enough to show that there is a class
  $d \in \Sc_1(S_g)$ such that $i(a,d) = i(d,c) = 1$. There exist representatives $\alpha$ and $\gamma$ of $a$ and $c$,
  with $\alpha \cap \gamma =
  \emptyset$. To find such a curve $d$, there are two cases to consider. If $\alpha \cup \gamma$ is nonseparating, by the change of coordinates, there is a  curve $\delta$
  intersecting both $\alpha$ and $\gamma$ once, and intersecting a (multicurve) representative of $[\beta]$ an odd number of times.
  Indeed, just note that $\alpha$ and $\gamma$ can be extended to a geometric symplectic basis $\nu$ for $S_g$.
  Let $\alpha'$ and $\gamma'$ be the curves intersecting $\alpha$ and $\gamma$ once respectively. The multicurve
  representative of $[\beta]$ is given by a union of $g$ curves $\beta_i$ around each torus neighborhood of a pair $\{\alpha_i,\alpha_i'\}$
  of $\nu$ with $i(\alpha_i,\alpha_i') = 1$. Call each such curve $\beta_i$ a \emph{local representative} for $[\beta]$.
  Thus local
  representatives of $[\beta]$ around $\{\alpha,\alpha'\}$ and $\{\gamma, \gamma'\}$ are given by
  $T_{\alpha}^k(\alpha')$ and $T_{\gamma}^j(\gamma')$, where $k,j \in \{0,1\}$
  depend on $[\beta]$ intersecting $\alpha'$ or $\gamma'$. Define $\delta$ by
  connecting $T_{\alpha}^{k'}(\alpha')$ and $T^{j}_{\gamma}(\gamma')$, where $k' \in \{0,1\}$ satisfy $k' = k + 1\mod 2$.

 If  $\alpha \cup \gamma$ is separating, then $\{a,c\}$ is a
 bounding pair, i.e. they bound a pair of subsurfaces of genus $g_1,g_2$ each with two boundary
 components. Applying the change of coordinates principle, there is
 a $d$ with $i(a,d) = 1 = i(a,c)$ and $d \in \Sc_1(S_g)$.
\end{proof}
\begin{proof}[Proof of \Cref{lemma:complex_con}] The goal is to modify the path given by \Cref{lemma:1_complex}
  to conclude the proof. It is enough to
show that if $a,c \in \Sc_1(S_g)$, with $i(a,c) =1$, then there are $b_1,b_2 \in \Sc_1(S_g)$ so
that $i(a,b_1) = i(b_1,b_2) = i(c,b_2) =1$ and whose pairwise sum, e.g. $[a] + [b_1]$, in $H_1(S_g,\Z/2\Z)$ is not $[\beta]$. In particular if $[a] + [c] \not= [\beta]$ then we are done.\\
Assume then that $[a] + [c] = [\beta]$. \Cref{fig:refining_path} shows the curves $b_1,b_2$.
\begin{figure}[H]
  \begin{center}
\begingroup%
  \makeatletter%
  \providecommand\color[2][]{%
    \errmessage{(Inkscape) Color is used for the text in Inkscape, but the package 'color.sty' is not loaded}%
    \renewcommand\color[2][]{}%
  }%
  \providecommand\transparent[1]{%
    \errmessage{(Inkscape) Transparency is used (non-zero) for the text in Inkscape, but the package 'transparent.sty' is not loaded}%
    \renewcommand\transparent[1]{}%
  }%
  \providecommand\rotatebox[2]{#2}%
  \newcommand*\fsize{\dimexpr\f@size pt\relax}%
  \newcommand*\lineheight[1]{\fontsize{\fsize}{#1\fsize}\selectfont}%
  \ifx\svgwidth\undefined%
    \setlength{\unitlength}{280.88855413bp}%
    \ifx\svgscale\undefined%
      \relax%
    \else%
      \setlength{\unitlength}{\unitlength * \real{\svgscale}}%
    \fi%
  \else%
    \setlength{\unitlength}{\svgwidth}%
  \fi%
  \global\let\svgwidth\undefined%
  \global\let\svgscale\undefined%
  \makeatother%
  \begin{picture}(1,0.30543153)%
    \lineheight{1}%
    \setlength\tabcolsep{0pt}%
    \put(0,0){\includegraphics[width=\unitlength,page=1]{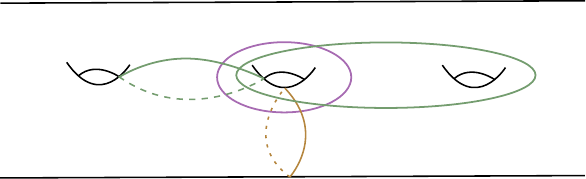}}%
    \put(0.36423675,0.2479813){\color[rgb]{0,0,0}\makebox(0,0)[lt]{\lineheight{1.25}\smash{\begin{tabular}[t]{l}$c$\end{tabular}}}}%
    \put(0.54554675,0.05932087){\color[rgb]{0,0,0}\makebox(0,0)[lt]{\lineheight{1.25}\smash{\begin{tabular}[t]{l}$a$\end{tabular}}}}%
    \put(0.81751182,0.24675625){\color[rgb]{0,0,0}\makebox(0,0)[lt]{\lineheight{1.25}\smash{\begin{tabular}[t]{l}$b_1$\end{tabular}}}}%
    \put(0.15842534,0.1009732){\color[rgb]{0,0,0}\makebox(0,0)[lt]{\lineheight{1.25}\smash{\begin{tabular}[t]{l}$b_2$\end{tabular}}}}%
  \end{picture}%
\endgroup%

  \end{center}
  \caption{Refining the path in $\Nc\CC_1(S_g)$ to lie on $\Nc_1(S_g)$.}\label{fig:refining_path}
\end{figure}

\end{proof}

Our next result characterizes the action of $\Mbeta$ on $\Nc_1(S_g)$, and thus gives us a new generating set
for $\Mbeta$.
\begin{lemma}[\bf{Generating set-Stabilizer}]\label{lemma:gen_pi1(Rg)}
  Let $g \geq 3$ and $[\beta]$ a nonzero class in $H_1(S_g,\Z/2\Z)$. $\Mbeta$ acts transitively on $\Sc_1(S_g)$ and on pairs of edges $(a_i,c_i)$ of $\Nc_1(S_g)$.
  In particular, for any $a \in \Sc_1(S_g)$, $\Mbeta$ is generated by the stabilizer of $a$ in $\Mbeta$ and
  any $h \in \Mbeta$, so that $(a,h^{-1}(a)) \in \Nc_1(S_g)$.
\end{lemma}
\begin{proof}
Let $\alpha_i,\gamma_i$ be representatives for $a_i,c_i$ in minimal position, and let $\delta_i$ be the boundary curve
  of the closed torus neighborhood $T_i$ of $\alpha_i \cup \gamma_i$. Let $P_i$ be the complementary subsurface bounded by
  $\delta_i$. By assumption, there exist multicurve representatives
  $\{\beta_1^i,\beta_2^i\}$ of $[\beta]$, with $\beta_1^i \subset T_i$ and $\beta_2^i \subset P_i$, furthermore $\beta_1^i$ intersects both $\alpha_i$
  and $\gamma_i$ only once. Let $f \in \Mod(S_g)$ with $f(\delta_1) = \delta_2$, and inducing
  homeomorphisms $f_T:T_1 \to T_2$ and $f_P:P_1
  \to P_2$. As the symplectic representation mod 2 is surjective, there exists $g_P \in \Mod(P_2)$ so that
  $(g_Pf_P)[\beta_2^1] = [\beta_2^2]$. On the other hand, note that $f_T$ maps $(\alpha_1,\gamma_1)$ to a $2$-chain in
  $T_2$. Thus, as $\Mod(T_2)$ acts transitively on $2$-chains, there is $g_T \in \Mod(T_2)$ such that $g_Tf_T(\alpha_1) = \alpha_2$ and $g_Tf_T(\gamma_1) = \gamma_2$. It
  follows that $g_T f_T$ maps $\beta_1^i$ to a curve intersecting $\alpha_2$ and $\gamma_2$ only once each, and so
  $g_Tf_T[\beta_1^1] = [\beta_1^2]$. The first claim follows by composing $f$ with the extensions of $g_T$ and $g_P$.

  Let $a \in \Sc_1(S_g)$ and $h \in \Mbeta$ so that $a$ and $h^{-1}(a)$ are connected by an edge in $\Nc_1(S_g)$. Then, the hypothesis
  of Lemma 4.10 of \cite{FM} are satisfied and the second claim follows.
\end{proof}


\subsection{Low-dimensional representations of $\Mbeta$}
The interplay between the two generating sets of $\Mbeta$ found in \Cref{subsec:complex-curves,subsec:conj} allows
us to conclude all but the last item of \Cref{theo:rigidity_group_d}.
\begin{proof}[Proof of (1)-\Cref{theo:rigidity_group_d}]
  Represent $[\beta]$ by a simple closed curve $b$, and let $a$ be a simple closed curve intersecting $b$
  transversely at one point. Let $R$ be the complement of an open annular neighborhood of $b$,
  then $R \cong S_{g-1}^2$ (A genus $g-1$ surface with \emph{two} boundary components), via the inclusion
  $R \hookrightarrow S_{g-1}$ there is a map
  $\Mod(R) \to \Mbeta$ and $\phi$ induces a representation $\phi_R:\Mod(R) \to \GL(m,\C)$.

  We claim that $\phi_R$ is trivial. For $m < 2g-2$ this follows from the results of Franks-Handel~\cite[Theorem 1]{Franks-Handel}, as the genus of $R$
  is at least $3$. Similarly for $m = 2g-2$, by Korkmaz~\cite[Theorem 2]{korkmaz}, $\phi_R$ is either trivial or conjugate to the standard symplectic
  representation $\psi:\Mod(R) \to \Sp(2g-2,\Z)$. Note that in either case, $\phi(T_b) = 1$ as $b$
  is separating in $R$. Let $d$ be the boundary of a regular neighborhood of $a \cup b$.
  Via the $2$-chain relation~(see \cite[Prop 4.12]{FM}),
  \[ \phi(T_a^2T_b)^4 = \phi(T_d) = 1,\]
  as $d \in R$ is separating. Thus,
  regardless of $\phi_R$, $\phi(T_a^2)$ is of order at most $4$ and by conjugation the same applies to any
  $\phi(T_{a'}^2)$ with $\hat{i}_2([a'],[\beta])=1$.

  Now suppose that $\phi_R$ is not trivial, then after conjugating $\phi$
  we can assume $\phi_R = \psi$. Consider two $k_i$-chains to each side of $b$ with $k_i$ odd, as in \Cref{fig:chain_conj}.
  \begin{figure}[H]
    \begin{center}
        \def \svgwidth{.9\textwidth}
    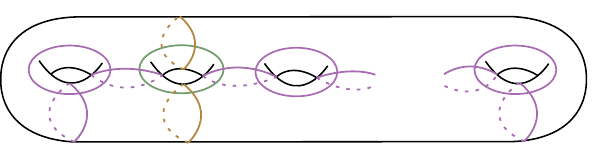
  \end{center}
  \caption{Complementary $k$-chains around $b$.}\label{fig:chain_conj}
\end{figure}

  Then, the $k$-chain relations~\cite[Prop 4.12]{FM} imply that:
  \[ (T_a^2T_{c_1}T_{c_2})^3 = (T_{a'}^2T_{c_3}\ldots T_{c_{2g-2}})^{2g-3} \]
  Let $\tilde{R} \subset R$, be the complement of a torus neighbohood of $d_1 \cup b$, then
  \[ \phi(\Mod(\tilde{R})) =  \Sp(2g-2,\Z). \]
  As $T_{d_1}^2$ commutes with any $f \in \Mod(\tilde{R})$, we find that $\phi(T_{d_1}^2) = \lambda \Id$ for some $\lambda \in \C^*$ and
  $\lambda^4 = 1$. By conjugation, $\phi(T_a^2) = \phi(T_{d_1}^2)$ for any $a$ with $\hat{i}_2([a],[\beta]) =1$.
  Furthermore, by assumption $\lambda^{2g-2} = 1$ for any $g$~(here we need the extra condition, $\Ima(\phi) \subset \SL(m,\C)$,
 for $g$ even).

  The $k$-chain relations, under $\phi$, induce the relation,
  \[ (\psi(T_{c_1})\psi(T_{c_2}))^3 = \lambda^{2g-2}(\psi(T_{c_3})\ldots \psi(T_{c_{2g-2}}))^{2g-3} \]

  A direct computation shows that $(\psi(T_{c_1})\psi(T_{c_2}))^3$ acts as $-\Id$ on $\{[c_1],[c_2]\}$,
  while any $\psi(T_{c_i})$ for $i > 2$ acts trivially, hence we reach a contradiction.

  Consequently, $\phi_R$ is trivial and so $\phi(T_c) = 1$ for any $c$ with $\hat{i}_2([c],[\beta]) = 0$.
  Let $c$ be a nonseparating SCC in $R$, so in particular $\phi(T_c) = 1$, meeting $a$ transversely at one point. Let $v = [T_c(a)]$, then
  $\hat{i}_2(v,[\beta]) =1$ and $v + [a] \not=[\beta]$. By \Cref{lemma:gen_pi1(Rg)},
  $\Mbeta$ is generated by $T^{-1}_c\in \Mod(R)$
  and the stabilizer of $a$. Hence, for any element $f \in \Mbeta)$, $\phi(f)$
  commutes with $\phi(T^2_a) = L_a$. For any other curve $a'$ with $\hat{i}_2([a'],[\beta])
  = 1$, $T^2_{a'}$ is conjugate to $T_a^2$ in $\Mbeta$.
  Thus, $\phi(T^2_{a'}) = L_a$ for all such $a'$. By \Cref{lemma:gen_dehn_pi1(Rg)}, $\phi(\Mbeta) = \dprod{L_a}$ and
  the theorem follows.
\end{proof}


\subsection{Lifting relations to $\Msigma$}\label{subsec:rel-msigma}

Let $\rho:S_{2g-1} \to S_g$ be the double cover with deck transform $\sigma$ and with monodromy
given by intersection with
$[\beta] \in H_1(S_g,\Z/2\Z)^*$.
By choosing lifts of elements of $\Mbeta$ to $\Msigma$, we can translate the results of \Cref{subsec:conj,subsec:complex-curves}
to $\Msigma$.

Dehn twists have distinguished lifts to $\Msigma$: let $a$ be an isotopy class of simple closed
curves in $S_g$. If $\hat{i}_2([a],[\beta]) = 0$, then $a$ has two disjoint (as $a$ is simple) and nonisotopic
lifts $\tilde{a}$ and $\sigma(\tilde{a})$ to $S_{2g-1}$. In this case, a lift of $T_a$ is given by the
\emph{multitwist} $T_{\tilde{a}}T_{\sigma(\tilde{a})} \in \Msigma$. Similarly, if $\hat{i}_2([a],[\beta]) =1$ let $\tilde{a}$
be the union (in any order) of the two simple paths lifting $a$ to $S_{2g-1}$. Then, a lift of
$T_a^2$ is given by $T_{\tilde{a}}$. The way we join both lifts of $a$ does not affect the lift as
both ways give isotopic loops. Furthermore, as $\sigma$ permutes the lifts of $a$, then $T_{\tilde{a}} \in \Msigma$.\\
With this notation, by lifting mapping classes from $\Mbeta$ to the cover $\Msigma$, we obtain the following.
\begin{cor}[\bf{Conjugation in $\Msigma$}]\label{cor:conj_Msigma}
  Let $a_1,a_2$ be a pair of isotopy classes of nonseparating simple closed curves
  in $S_g$.
  \begin{enumerate}
  \item If $\hat{i}_2([a_1],[\beta]) = \hat{i}_2([a_2],[\beta]) = 1$, then $T_{\tilde{a}_1}$ and $T_{\tilde{a}_2}$
    are conjugate in $\Msigma$.
  \item If $\hat{i}_2([a_1],[\beta]) = \hat{i}_2([a_2],[\beta]) = 0$, and either for each $i$
    $[a_i] \not=[\beta]$ or for each $i$ $[a_i] = [\beta]$, then $T_{\tilde{a}_1}T_{\sigma(\tilde{a}_1)}$ and
    $T_{\tilde{a}_2}T_{\sigma(\tilde{a}_2)}$ are conjugate in $\Msigma$.
  \end{enumerate}
\end{cor}


Similarly, \Cref{lemma:gen_dehn_pi1(Rg),lemma:gen_pi1(Rg)} imply the following results.
\begin{cor}[\bf{Generating set $\Msigma$-twists}]\label{cor:gen_twist_sigma}
Let $g \geq 2$, and $\sigma$ the deck transform of the double cover $p:S_{2g-1} \to S_g$ associated to $[\beta] \in H_1(S_g,\Z/2\Z)^*$.
$\Msigma$ is generated by
\begin{equation*}
  \begin{aligned}
  \{\sigma\} \cup
    \{T_{\tilde{c}}(T_{\sigma(\tilde{c})})^{\xi(c)}: c &\mbox{ nonseparating SCC in $S_g$ and } \xi(c) \in \{0,1\},\\
    &\mbox{ with } \xi(c)= \hat{i}_2([c],[\beta]) + 1 \mod 2\}.
  \end{aligned}
  \end{equation*}
\end{cor}
\begin{cor}[\bf{Generating set $\Msigma$-stabilizer}]\label{cor:lift_stab}
  Let $a$ be an isotopy class in $\Sc_1(S_g)$. $\Msigma$ is generated by the following elements: $\sigma$, $f$ such that
  $fT_{\tilde{a}}f^{-1} = \sigma^i T_{\tilde{a}}$, and $h$ such that $(\overline{h}^{-1}(a),a) \in \Nc_1(S_g)$.
  Where $\overline{h}\in \Mbeta$ is the projection of $h \in \Msigma$.
\end{cor}
\medskip
\begin{remark}[\bf{Relations in $\Msigma$}] Note that for any proper subsurface $S \subset S_g$, there is a lift $\Mod(S) \cap \Mbeta \to \Msigma$ defined as follows: Since $f \in \Mod(S)$ fixes a point $p$ in $S_g \setminus S$, this implies the existence of a well defined
choice of lift to $\Msigma$ by requiring the map to fix a lift of $p$.

On the other hand, it is not possible to lift all of $\Mbeta$. A way to see this is to note that $\Prym_*$ surjects
onto $\Sp(2g-2),\Z)$ and thus a lift would give an infinite image representation of $\Mbeta$ contrary
to \Cref{theo:rigidity_group_d}.

In fact, there is an explicit relation that cannot hold in $\Msigma$. Let $\tilde{R}$ be the subsurface defined in the
proof of \Cref{theo:rigidity_group_d}. $\Prym_*$ acts as the symplectic representation on the lift of $\Mod(\tilde{R})$,
while $\Prym_*(T_{\tilde{a}}) = 1$ for any $a$ with $\hat{i}_2([a],[\beta]) = 1$.
Hence, the $k$-chain relations used in the proof of \Cref{theo:rigidity_group_d} \emph{cannot hold} in $\Msigma$ and so
\[ (T_{\tilde{a}}T_{\tilde{c}_1}T_{\sigma(\tilde{c}_1)} T_{\tilde{c}_2}T_{\sigma(\tilde{c}_2)})^3=
\sigma(T_{\tilde{a}'}T_{\tilde{c}_3}T_{\sigma(\tilde{c}_3)} \ldots T_{\tilde{c}_{2g-2}}T_{\sigma(\tilde{c}_{2g-2})})^{2g-3}\]
\end{remark}


\subsection{Low dimensional representations of $\Msigma$}\label{subsec:low-dim-msigma}
The relations in $\Msigma$ described on \Cref{subsec:rel-msigma} , and the proof of \Cref{theo:rigidity_group_d}.1
imply the first item in \Cref{theo:rigidity_group}.
\begin{proof}[Proof of \cref{theo:rigidity_group}.1]
  Let $a \in \Sc_1(S_g)$, and $T_{\tilde{a}}$ the lift of $T_a^2$ to $\Msigma$. By liftying $\Mod(R)$,
  where $R$ is the subsurface used in the proof of \Cref{theo:rigidity_group_d}.1, it follows that $\phi(T_{\tilde{c}}T_{\sigma(\tilde{c})}) = 1$.
    Thus, by the lifted $k$-chain relation, $\phi(\sigma) \in \dprod{\phi(T_{\tilde{a}})}$.
    One then concludes by the same argument as in \Cref{theo:rigidity_group_d}.1, replacing \Cref{lemma:gen_pi1(Rg)} with \Cref{cor:lift_stab}.

\end{proof}

To tackle the second item of \Cref{theo:rigidity_group} with $m = 2g-2$, we use the results of Dey et.al.~\cite[Theorem 1]{dey2021generating} to get an explicit finite
generating set for $\Msigma$. We have the following corollary of their theorem 1.
\begin{cor}[\bf{Finite generating set}]\label{cor:fin_gen_msigma}
  Let $c_i,a_i,b_i$ be the curves in the top of \Cref{fig:dey}. Let $N(a_1 \cup b_1)$ be a Torus neighborhood of $a_1 \cup b_1$.
  $\Msigma$ is generated by $\sigma$ and chosen
  lifts of
  \[ S \cup \{F_2,\ldots,F_{g-1}\} \cup \{T_{a_2},T_{b_2},\ldots, T_{a_g},T_{b_g},T_{c_1},\ldots,T_{c_{g-1}}\} \]
  Where $F_i$ are the bounding pairs given at the bottom of \Cref{fig:dey}, and $S$ is a generating set for the subgroup of
  $\Mod(N(a_1 \cup b_1))$ fixing $[\beta] = [b_1]$ mod 2.
\end{cor}

\begin{figure}[H]
  \begin{center}
    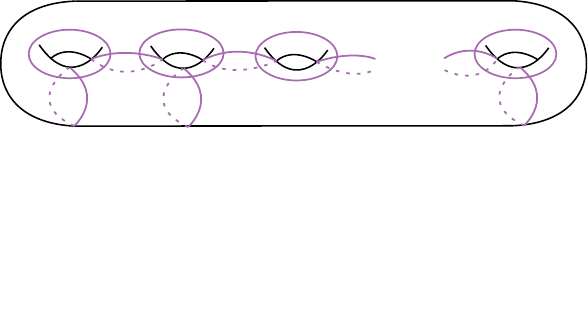
  \end{center}
  \caption{Top: Curve generators for $\Mbeta$. Bottom: Torelli generators for $\Mbeta$.}\label{fig:dey}
\end{figure}

\begin{proof}[Proof of \Cref{theo:rigidity_group}.2]
  Let $\phi:\Msigma \to \GL(2g - 2,\C)$ be any nonfinite representation.
  Represent $[\beta]$ by
  $[b_1]$. Let $R$ be the complement of an annular neighborhood of $b_1$. Any $f \in \Mod(R)$ fixes $b_1$ \emph{pointwise}.
  Thus, we can lift $\Mod(R)$ to
  $\Msigma$ fixing the lift $\tilde{b}_1$ of $b_1$ pointwise, in particular any Dehn twist $T_c \in \Mod(R)$
  lifts to a multitwist $T_{\tilde{c}}T_{\sigma(\tilde{c})}$. It follows that $\phi$ induces a representation $\phi_R:\Mod(R) \to \GL(2g - 2,\C)$.
  As $\phi$ is nonfinite, we must have $\phi_R$ nontrivial thus after a conjugation we can assume that
  $\phi_R = \psi$, where $\psi$ is the standard \emph{symplectic representation}. Note that
  $[\tilde{c}] - \sigma_*[\tilde{c}]$ for SCCs $c \in R$, generate $H_1(S_{2g-1},\Z)^-$, and thus $\Prym_*$ acts as
  $\psi$ on this lift of $\Mod(R)$, thus $\phi_R = \Prym_*|_R$.

  It remains
  to check the action of $\phi$ on the other generators of $\Msigma$ coming from \Cref{cor:fin_gen_msigma}.
  Let $T = N(a_1 \cup b_1)$ be a torus neighborhood of $a_1 \cup b_1$, and $\tilde{R} \subset R$ the
  the complementary subsurface. As any mapping class in $\Mod(T)$ fixes $\tilde{R}$ pointwise, there is a lift $\widetilde{\Mod(T)}$ of $\Mod(T)$ to $\Msigma$,
  so that the lift of each element fixes both lifts of $\tilde{R}$ to $S_{2g-1}$ pointwise. In particular, any lift $\tilde{f}$
  of $f \in \Mod(T)$ commutes with lifts of $\Mod(\tilde{R})$. As $\phi_R(\tilde{R}) = \Sp(2g-2,\Z)$,
  it follows that $\phi(\tilde{f})$ is a scalar, for any $f \in \Mod(T)$. In particular $T_{\tilde{a}_1} = \lambda \in \C^*$,
  and so by \Cref{cor:conj_Msigma}, the same holds for any $T_{\tilde{a}}$ with $\hat{i}_2([a],[\beta]) = 1$. $\Prym_*$ acts trivially
  on this lift of $\Mod(T)$, thus
  \[ \phi_T. \Prym_*|_{\widetilde{\Mod(T)}}^{-1} \in \C^*.\]

  Next, note that by the chain-relations each bounding pair $F_i$ can be expressed as,
  \[ F_i = (T_{a_1}^2T_{c_1}T_{a_2}\ldots T_{c_{i-1}}T_{a_i})^{2i-1}T_{d_i}^{-2} \]
  Where $T_{d_i}$ is one of the curves of the bounding pair. So for each lift $\tilde{F}_i$,
  \[ \phi(\tilde{F}_i).\Prym_*(\tilde{F}_i)^{-1} \in \C^*. \]

  Lastly, $\phi(\sigma)$ commutes with any element of $\phi(\Msigma)$. Thus as \[\Ima(\phi) = \Sp(2g-2,\Z),\] $\phi(\sigma)$ must be a scalar.

  It follows that for any $f \in \Msigma$, $\phi(f)\Prym_*(f)^{-1} = \lambda(f) \in \C^*$. We claim that
  $f \mapsto \lambda(f)$ is a homomorphism. Indeed,
  \[ \phi(fg)\Prym_*(fg)^{-1} = \phi(f)\phi(g) \Prym_*(g)^{-1}\Prym_*(f)^{-1}=
    \lambda(f).\lambda(g) \]

  In particular $\lambda:\Msigma \to \C^*$ must be cyclic of order at most $4$.
\end{proof}
\begin{proof}[Proof of \Cref{theo:rigidity_group_d}.2]
Let $g \geq 4$ be even, \Cref{theo:rigidity_group} gives us an example of an infinite representation $\phi:\Mbeta \to \GL(2g-2),\C)$. Let $\chi:\Msigma \to \C^*$, with $\chi(\sigma) = -1$. Then, $\phi$ is induced by,
\[ \tilde{\phi}:\Msigma \to \GL(2g-2,\Z) \ \ , \ \ f \to \chi(f) \Prym_*(f).\]

The same argument as in the of the proof of \Cref{theo:rigidity_group}.2, replacing $\Prym_*$ with $\phi$,
concludes the proof of \Cref{theo:rigidity_group_d}.
\end{proof}


\subsection{Symplectic representations}
\Cref{subsec:low-dim-msigma} considered representations $\phi:\Msigma \to \GL(m,\C)$. The results extend easily to cases
where the image is contained in $\Sp(2h,\Z)$. Let $\Delta(2h,\Z)$ be the subgroup of $\GL(2h,\Z)$ fixing the
symplectic form up to sign.
\begin{remark} Note that any element of $\Delta(2h,\Z)$ is of the form $Z^i A$ for $A \in \Sp(2h,\Z)$, and
  $Z = \begin{pmatrix} \Id & 0 \\ 0 & -\Id \end{pmatrix}$.\end{remark}

\begin{cor}\label{cor:symp_rep}
  Let $g \geq 4$ and $h \leq (g-1)$. Let $\phi:\Msigma \to \Sp(2h,\Z)$ be a homomorphism, then
  \begin{enumerate}
  \item If $h < g-1$, then $\Ima(\phi)$  is a quotient of $\Z/4\Z$.
  \item If $h = g-1$, then either $\Ima(\phi)$ is a quotient of $\Z/4\Z$
    or up to a conjugation in $\Delta(2g - 2,\Z)$ is of the form:
    \[ f \to \chi(f) \Prym_*(f) \]
    where $\chi:\Msigma \to \{-\Id,+\Id\}$ is any group homomorphism.
  \end{enumerate}
\end{cor}
\begin{proof} The first item follows directly from \Cref{theo:rigidity_group}. For the second, assume that the image is not finite as the finite case
  follows directly as well. By \Cref{theo:rigidity_group}
  there is a matrix $A \in \GL(2g - 2,\C)$ such that $\Psi = A \phi A^{-1}$ is of the desired form. Indeed,
  Let $R$ be the subsurface in the proof of \Cref{theo:rigidity_group}, then  note that $A$ is chosen so that
  $\phi_R:= \phi|_{\Mod(R)}$ is the standard symplectic representation. In particular the conjugation by $A$, $C_A:\Sp(2g - 2,\Z) \to
  \Sp(2g - 2,\Z)$, is an automorphism. Reiner~\cite[Theorem 3]{Reiner}
  showed that all such automorphisms come from conjugation in $\Delta(2g - 2,\Z)$. With this remark
  in place, the proof of \Cref{theo:rigidity_group} goes through without modifications. Importantly, the
  image of $\chi$ lies on the centralizer of $\Ima(\phi) = \Sp(2g-2,\Z)$, hence the result.
\end{proof}





\section{Holomorphic results}\label{sec:holo}
The aim of this section is to complete the proof of \Cref{theo:prym_rigid}.
Fix $g \geq 4$ and $h\leq g-1$, and consider a holomorphic map of complex orbifolds
\[ F: \CR_g \to \A_h.\]

The proof follows Farb's strategy which consists of 6 steps~\cite[p 2-3]{farb2021global}.
\subsection{Farb's proof strategy}\label{sec:farb_proof}
\begin{enumerate}
  \item[Step 1:] As we will see below (\cref{subsubsec:trivial-prym,subsubsec:f-prym}), the results of \Cref{sec:Topology} quickly
reduce the statement to the case of $h =g-1$ and $F$ homotopic to $\Prym$.
\item[Step 2:] The important observation is that $\h_{g-1}$ is a bounded symmetric domain, so by a criterion of Borel-Narasiman~\cite{Borel-Narasimhan}
it is enough to check that $F$ and $\Prym$ agree at \emph{point} $x \in \CR_g$. The remaining steps consist
on finding such an $x$.
\item[Step 3-4:] Restrict the homotopy $F_t$ between $\Prym$ and $F$ to
a curve $C$. (\textbf{Step 3}) Improve the homotopy to a geodesic homotopy, and following an argument of Antonakouidis-Aramayona-Souto~\cite{AAS}
and a Wirtinger-type inequality improve the homotopy even further to $F_t|_C$ be holomorphic at each $t$. (\textbf{Step 4}) Finally,
using a theorem of Kobayashi--Ochiai~\cite[Theorem 2']{KO}, improve $F_t|_C$ to be algebraic at each $t$. These steps carry over
without any modification as they depend on the target of the map being $\A_g$ (or covers of).
\item[Step 5:] Find a $\A_{g-1}$-\emph{rigid curve} $C \subset \CR_g$ (see \cref{subsubsec:rigid_curve}), so
  that $\Prym$ is an isolated point in $\Mor(C,\A_{g-1})$.
  \item[Step 6:] $F_t|_C$ is constant, so in particular $F$ and $\Prym$ agree on $C$ and we are done.
\end{enumerate}
In our case, to avoid orbifold issues, steps 3-6 would be carried out on a suitable finite cover of $\CR_g$.
An alternative proof, suggested to Farb by Richard Hain and replacing steps 3-6 by a variation of Hodge structures argument~\cite[Section 3]{farb2021global}, is also given.
\subsection{Case $h < g-1$}\label{subsubsec:trivial-prym}
By \Cref{theo:rigidity_group}, there is a finite cover $\tilde{\CR}_g \to \CR_g$ such that the induced map $\tilde{F}: \tilde{\CR}_g \to \A_h$ satisfies $\tilde{F}_*:\pi_1^{\orb}(\tilde{\CR}_g)  \to \Sp(2h,\Z)$ is trivial. Thus, $\tilde{F}:\tilde{\CR}_g \to
\A_h$ lifts to a holomorphic map $G:\tilde{\CR}_g \to \h_h$, in particular $\tilde{F}$ is homotopic to a constant map, and the same holds for any such continuous $F$ (cf.~\Cref{cor:continuous_rigidity}). As $\tilde{\CR}_g$ is a finite (branched) cover
of $\CR_g$ it is also a quasiprojective variety, and as $\h_h$ is a bounded domain it follows that $G$ is constant. The same argument gives a proof of \Cref{theo:triviality_rg}.

In what follows we will assume that $F_*:\Msigma \to \Sp(2g-2,\Z)$ has infinite image, as otherwise the same argument shows that $F$ is constant.

\subsection{Case $h = g-1$}
\subsubsection{The Prym map}\label{subsubsec:prym_map}
Let $X$ be a smooth genus $g$ complex curve. Any nonzero $\theta \in H^1(X,\Z/2\Z)$ defines an unbranched double cover
\[ p:Y \to X, \]
with deck transform $\sigma$, and where $Y$ is a curve of genus $2g-1$.
The order 2 action of $\sigma^*$ on $\Omega^1(Y)$ induces a splitting
\[ \Omega^1(Y) = \Omega^1(Y)^+ \oplus \Omega^1(Y)^- \]
corresponding to the $\pm 1$ eigenspaces of $\sigma^*$. Similarly, the action of $\sigma_*$ on $H_1(Y,\Z)$
has two distinct subspaces\footnote{But this is \emph{not} a splitting of
  $H_1(Y,\Z)$.}, $H_1(Y,\Z)^+$ and $H_1(Y,\Z)^-$.
The \emph{Prym variety} associated to $(X,\theta)$ is defined as
\[ \Prym(X,\theta) := \frac{(\Omega^1(Y)^-)^\vee}{H_1(Y,\Z)^-}.\]

$\Prym(X,\theta)$ is a subtorus of $\Jac(Y)$, and the restriction of the principal polarization from
$\Jac(Y)$ (given by the intersection pairing on $H_1(Y,\Z)$) to $\Prym(X,\theta)$ induces twice a
principal polarization. In particular, $\Prym(X,\theta)$ is a PPAV of dimension $g-1$. This description is in fact, equivalent to the one given in the introduction. The map
\[ \Jac(p):\Jac(Y) = \frac{\Omega^1(Y)^\vee}{H_1(Y,\Z)} \to \Jac(X) = \frac{\Omega^1(X)^\vee}{H_1(X,\Z)}, \]
is induced by the dual of the pullback $p^*:\Omega^1(X) \to \Omega^1(Y)$. Hence,
\[ \ker \Jac(p) = \frac{(\Omega^1(Y)^-)^\vee}{H_1(Y,\Z)^-} = \Prym(X,\theta).\]

\medskip
\noindent{\textbf{The Prym period matrix.}} Recall that $\Fix(\sigma) \subseteq \teich_{2g-1}$ is
the space of \emph{marked} curves $[(Y,\phi)]$ with a fixed-point free holomorphic involution $\sigma_\phi$ isotopic
to $\phi \sigma \phi^{-1}$.
Consider $[(Y,\phi)] \in \Fix(\sigma) \subset \teich(S_{2g-1})$.
Let $\{a_i,b_i\}$ for $i=0,\ldots,2g-2$, be a geometric symplectic basis for $S_{2g-1}$ such
that
\[ \sigma(b_i) = b_{i+g-1} \ \ , \ \sigma(a_i) = a_{i+g-1} \ \ \ i = 1,\ldots, g-1 \]
Let $\omega_i$ be a basis for $\Omega^1(Y)$ dual to $\{\phi(a_i)\}$, and let $u_i := \frac{\omega_i - \omega_{i+g-1}}{2}$
for $i = 1,\ldots, g-1$. Then, $\{u_i\}$ is a basis for $\Omega^1(Y)^{-1}$, and in fact $\{u_i\}$ is dual to
$\{\phi(a_i) - \phi(a_{i+g-1})\}_{ i \geq 1}$. Importantly, the marking $\phi:S_{2g-1} \to Y$ allows us to choose such
$\{u_i,\phi(a_i),\phi(b_i)\}$ globally over $\Fix(\sigma)$ (see also \cite{Farkas-Rauch,Mumford}
and \cite[Section 3.1]{Sato}).
Then,
\[ \tau = \left( \int_{\phi(b_j) - \phi(b_j + g-1)} u_i \right) \in \h_{g-1}.\]
Let $\widetilde{\Prym}:\Fix(\sigma) \to \h_{g-1}$ be $[(Y,\phi)] \to \tau$.
Moreover, if the (normalized) period matrix of $Y$ with respect to $\{\phi(a_i),\phi(b_i)\}$ and $\{\omega_i\}$ is given by:
\[ \left(\int_{\phi(b_j)} \omega_i\right)_{0\leq i,j \leq 2g -2}  = \begin{pmatrix}
    * & * & *\\
    * & B & C^T \\
    * & C & D
  \end{pmatrix}
\]
Then, $\tau = B - C$ and in particular $\widetilde{\Prym}$ is holomorphic.
Similarly, by a direct computation one can check that $\widetilde{\Prym}$ is $\Prym_*$-equivariant\footnote{
  With respect to the action of $\Msigma$ on $\Fix(\sigma)$, $f\cdot [(Y,\phi)] \to [(Y,\phi \circ f^{-1})]$,
  we find equivariance, but with $\Sp(2g - 2,\Z)$ acting by $\begin{pmatrix}\alpha & \beta \\ \gamma
    &\delta \end{pmatrix}\cdot \tau \to (\alpha \tau - \beta)(-\gamma \tau + \delta)^{-1}$.}
and lifts $\Prym$
\begin{center}
\begin{tikzcd}
  &\Fix(\sigma) \ar[r,"\widetilde{\Prym}"] \ar[d, ""] & \h_{g-1} \ar[d]\\
  &\CR_g \ar[r,"\Prym"] & \A_{g-1}
\end{tikzcd}
\end{center}

\subsubsection{$F$ homotopic to $\Prym$}\label{subsubsec:f-prym}
Let $F: \CR_g \to \A_{g-1}$ be a nonconstant holomorphic map. Then, by \Cref{cor:symp_rep}
there is an $A \in \Delta(2g-2,\Z)$ and $\chi:\Msigma \to \{\pm \Id \}$ such that
\[ A F_*A^{-1} = \chi \Prym_* . \]
If $A \in \Sp(2g,\Z)$, it follows that there is a lift $\tilde{F}:\Fix(\sigma) \to \h_{g-1}$ which is equivariantly
homotopic to $\widetilde{\Prym}$. Indeed, this is
because $\chi$ acts trivially on $\h_{g-1}$ so both $F_*$ and $\Prym_*$ factor trough the \emph{same} representation
$\widehat{\Prym}_*: \Mbeta \to \PSp(2g - 2,\Z)$. In fact,
the homotopy can be chosen to be a straight-line homotopy, i.e. a homotopy through geodesics, as $\h_{g-1}$ has a K\"ahler metric of nonpositive
curvature under which the action of $\Sp(2g - 2,\Z)$ is by isometries.

The case in which $A = Z B$ for $B \in \Sp(2g - 2,\Z)$, and $Z = \begin{pmatrix}
  \Id & 0 \\ 0 &-\Id \end{pmatrix}$
can be ruled out as follows: Consider the map $G: \h_{g-1} \to \h_{g-1}$, given by $\tau \to - \overline{\tau}$,
then $G$ is $Z$-equivariant and anti-holomorphic. In particular there is a lift $\tilde{F}$ of $F$ such
that $F_G:= G \circ \tilde{F} $ is equivariantly homotopic to $\widetilde{\Prym}$, hence we have a holomorphic
map $\Prym$ homotopic to an antiholomorphic map $F_G$. As $\CR_g$ contains a smooth holomorphic closed curve $X$ this is imposible
(see also \cref{subsubsec:VHS}).
Indeed, let $\omega$ be the K\"{a}hler form on $\A_{g-1}$, then restricting the maps to $X$, due to $F_G$ being antiholomorphic
 and $\omega_X$ being the volume form on $X$, $F_G^*(\omega) = f_1 \omega_X$
where $f_1 \leq 0$. On the other hand, $\Prym^*(\omega) = f_2\omega_X$ for $f_2 \geq 0$. By Stokes' theorem we then find $f_1 = f_2 = 0$.
Hence $\Prym$ is constant over $X$, and as we can choose $X$ generically we reach a contradiction.

\subsubsection{$\psi$-structures}
The arguments in \Cref{subsubsec:trivial-prym,subsubsec:f-prym} show that the results of \Cref{sec:Topology} imply that
if $F:\CR_g \to \A_{g-1}$ is nonconstant, then it is homotopic to $\Prym$.
One could carry out the next steps in Farb's proof~\cite{farb2021global} under this setting, but the orbifold issues become cumbersome
at the last step (existence of rigid curves).

To circumvent these issues we first pass to a finite cover of $\CR_g$. Let $[\beta] \in H_1(S_g,\Z/2\Z)^*$,
and $\rho : S_{2g-1} \to S_g$ be the double cover induced by the map $\pi_1(S_g) \to \Z/2\Z$, given by $\gamma \to \hat{i}_2([\gamma],[\beta])$.
Next, let $\hat{S}_L \to S_{2g-1}$ be the cover induced by the surjection $\pi_1(S_{2g-1}) \to H_1(S_{2g-1},\Z/L\Z)$ for $L\geq 3$
sending $\gamma \mapsto [\gamma]$, its homology class mod $L$.
Then the composite cover is induced by the map
\[ \psi: \pi_1(S_g) \to \frac{\pi_1(S_g)}{\dprod{\pi_1(S_{2g-1})',\pi_1(S_{2g-1})^L}} =: G.\]
Where $\pi_1(S_{2g-1})^L := \{\gamma^L:\gamma \in \pi_1(S_{2g-1}) \}$ and $\pi_1(S_{2g-1})'$
is the commutator subgroup $[\pi_1(S_{2g-1},\pi_1(S_{2g-1})]$ of $\pi_1(S_{2g-1})$.
Let $\Gamma_g[\psi]$ be the stabilizer of $\psi$ (as an exterior homomorphism) on $\Mod(S_g)$ and consider the finite cover
\[ \hat{\CR}_g[\psi] := \teich(S_g)/\Gamma_g[\psi] \]
of $\hat{\CR}_g$, given by attaching to each curve a level $\psi$-structure \cite[p.511]{arbarello2011geometry}.

\begin{defn}[\bf{Prym Level-L structures}]
   For any integer $L \geq 0$, we define
   \[ \Mod(S_{2g-1},\sigma)[L] = \Prym_*^{-1}(\Ker\{\Sp(2g - 2,\Z) \to \Sp(2g - 2,\Z/L\Z)\}) \]
 \end{defn}
\begin{remark} Unlike $\Mod(S_g)[L]$ (the \emph{level $L$ congruence subgroup}~\cite[Ch 6.4.2]{FM}), the group $\Msigma[L]$ contains torsion for $L \geq 3$. This is because the kernel of $\Prym_*$ contains
  torsion: a lift of the hyperelliptic involution from $S_g$ to the cover $S_{2g-1}$ acts under $\Prym_*$
  in the same way as $\sigma$.
\end{remark}
Importantly for us $\Gamma_g[\psi]$ satisfies the following properties:
\begin{lemma}\label{lemma:psi}\mbox{}
  \begin{enumerate}
  \item   $\Gamma_g[\psi] \subset \Mod(S_g)[L] \cap \pi(\Msigma[L])$.
  \item Let $b$ be a SCC representative of $[\beta]$, and let $a$ be a SCC intersecting $b$ transversely at one point.
    Let $R$ be the complement of a torus neighborhood of $a \cup b$. Then $\Mod(R)[L] \subset \Gamma_g[\psi]$.
    \end{enumerate}
\end{lemma}
\begin{proof} Pick a basepoint $x \in S_g-R$ and $\tilde{x}$ the corresponding basepoint for $\pi_1(S_{2g-1})$.
  \begin{enumerate}
  \item
    We will show that in fact, for any $f \in \Gamma_g[\psi]$, there is a lift $\tilde{f}:S_{2g-1} \to S_{2g-1}$
  such that $\tilde{f}$ acts trivially on $H_1(S_{2g-1},\Z/L\Z)$ and $f_* \in \Mod(S_g)[L]$.
  The latter follows easily as we have a surjection
  \[ G = \frac{\pi_1(S_g)}{\dprod{\pi_1(S_{2g-1})',\pi_1(S_{2g-1})^L}} \twoheadrightarrow \frac{\pi_1(S_g)}{\dprod{\pi_1(S_g)',\pi_1(S_g)^L}} =  H_1(S_g,\Z/L\Z). \]
  Such that the projection $\pi_1(S_g) \to H_1(S_g,\Z/L\Z)$ factors through $\psi$. Similarly
  \[ G \twoheadrightarrow \pi_1(S_g)/(\pi_1(S_{2g-1})),\] and so $f \in \Mbeta$.

  Finally, let $\tilde{\gamma} \in \pi_1(S_{2g-1})$ and $\gamma$ its image in $\pi_1(S_g)$. Let $f \in \Gamma_g[\psi]$
  and pick a representative $\phi$ fixing $x$. Then, let $\tilde{\phi}$ be the lift fixing $\tilde{x}$.
  By assumption there is a loop $\beta \in \pi_1(S_g,x)$, independent of $\gamma$, such that $\tilde{\phi}(\gamma).\widetilde{\beta \gamma^{-1} \beta^{-1}} \in \dprod{\pi_1(S_{2g-1})',\pi_1(S_{2g-1})^L}$.
  The two possible lifts for $\beta \gamma^{-1} \beta^{-1}$ starting at $\tilde{x}$ are given by $\tilde{\beta}\sigma^{i}\tilde{\gamma}^{-1}\tilde{\beta}^{-1}$,
  for $i=0,1$ depending if the lift $\tilde{\beta}$ of $\beta$ starting at $\tilde{x}$ is  a loop or not.
  Hence, either $\sigma \tilde{\phi}$ or $\tilde{\phi}$ acts trivially on $H_1(S_{2g-1},\Z/L\Z)$ and the result follows.

\item Let $f \in \Mod(R)[L]$, then $f$ has a representative $\phi$ fixing the complement of $R$ pointwise.
  In particular $f_*(a) = a$ and $f_*(b) = b$. Furthermore, as $R$ lifts to $S_{2g-1}$, $\phi$ has a lift
  $\tilde{\phi}$ acting trivially on $H_1(S_{2g-1},\Z/L\Z)$
  so that $\phi(\gamma).\gamma^{-1} \in \dprod{\pi_1(S_{2g-1})',\pi_1(S_{2g-1})^L}$. Thus, $f_*$ will
  fix $\psi$ over $\pi_1(S_{2g-1})$ and $a$, and so the result follows.

\end{enumerate}

\end{proof}

\begin{remark} Note that as $\sigma \notin \Mod(S_{2g-1},\sigma)[L]$ for $L \geq 3$, $\Gamma_g[\psi]$ has a lift $\Lambda_g[\psi] \subset \Msigma$. Denote by $\Prym[\psi]$ the restriction $\Prym[\psi]|_{\Lambda_g[\psi]}$. Then, \Cref{lemma:psi}.2 implies that,
\[ \Prym[\psi]:\Lambda_g[\psi] \twoheadrightarrow \Sp(2g - 2,\Z)[L].\]
Denoting by $\CR_g[\psi] = \Fix(\sigma)/\Lambda_g[\psi]$,
it follows that $\CR_g[\psi] \cong \hat{\CR}_g[\psi]$, so to avoid excessive notation we will denote $\hat{\CR}_g[\psi]$ by
$\CR_g[\psi]$. Similarly, we will consider $\Prym[\psi]$ as a map from $\Gamma[\psi]$. Furthermore, $\Gamma_g[\psi]$ is torsion free and so
$\CR_g[\psi]$ is a $K(\pi,1)$-manifold, and a non-Galois\footnote{There exists $T_d$ along a separating curve $d$ so that its
  lift to $S_{2g-1}$ acts nontrivially on $H_1(S_{2g-1},\Z/L\Z)$ and so $T_d$ not in $\Gamma_g[\Psi]$.} cover of $\M_g$ with fundamental group \emph{isomorphic} to $\Gamma_g[\psi]$.
\end{remark}

Let $\A_{g-1}[L]$ be the moduli space of PPAV with level $L$-structure. It follows that any nonconstant holomorphic map $F:\CR_g \to \A_{g-1}$ induces a holomorphic map $F[\psi]:\CR_g[\psi] \to \A_{g-1}[L]$, with
$F[\psi]_*:\Gamma_g[\psi] \to \Sp(2g - 2,\Z)[L]$ equal to $\Prym[\psi]_*$.
 Consequently, $F[\psi] \sim \Prym[\psi]$ and Steps 3-4 in Farb's proof~\cite{farb2021global} (see also \Cref{sec:farb_proof}) carry over without modification (In fact, one could have also lifted the period map to $\M_g[L]$ in order to prove
its rigidity for $g \geq 3$.). It follows that for a curve $C \subset \CR_g[\psi]$ there exists a homotopy
$F_t$ between $F[\psi]$ and $\Prym[\psi]$, which is algebraic at each $t$.

\subsubsection{Finishing the proof via Rigid curves}\label{subsubsec:rigid_curve}
To conclude the proof of \Cref{theo:prym_rigid} we just need to show that $\A_{g-1}[L]$-rigid curves exist
in our setting. More precisely, we will show that there exists a curve $i: C \hookrightarrow \CR_g[\psi]$ so that:
\[ \Prym[\psi] \circ i : C \to \A_{g-1}[L] \]
is rigid. As in Farb's case this is done by finding a family satisfying Saito's criterion\cite[Thm 8.6]{saito1993classification}.
The goal is to find a noncompact curve $C$ so that the family over $C$ has irreducible monodromy, and infinite
order monodromy around some of the boundary points. The argument follows closely \cite[Step 5]{farb2021global}.

Let $\overline{\M_g}$ denote the Deligne-Mumford compactification of $\M_g$. Let $\overline{\CR_g[\psi]}$
be the compactification of $\CR_g[\psi]$, given by the normalization of $\overline{\M_g}$
on the function field of $\CR_g[\psi]$, in particular $\overline{\CR_g[\psi]} \to \overline{M_g}$ is
a finite branched cover and $\overline{\CR_g[\psi]}$ is a projective variety. Thus, we can assume
that $\CR_g[\psi] \subset \overline{\CR_g[\psi]} \subset \Pro^N$ for some $N$. As $\dim(\CR_g[\psi]) = 3g-3$,
by Bertini's theorem, the intersection of $\CR_g[\psi]$ with $3g-4$ generic hyperplanes is a smooth curve
$C \subset \CR_g[\psi]$.

By the Lefschetz hyperplane theorem for quasi-projective varieties, the inclusion
$C \xhookrightarrow{} \CR_g[\psi]$ induces a surjection
$\pi_1(C) \twoheadrightarrow \pi_1(\CR_g[\psi]) = \Gamma_g[\psi]$.

Let $Z$ be the unique codimension $1$-stratum of $\partial \M_g$ containing curves with nodes coming
from pinching a unique nonseparating loop, and let $Z[\psi]$ be its preimage on $\partial \CR_g[\psi]$.

Let $R$ be as in item $2$ of \Cref{lemma:psi}, and let $\X \to \Delta$ be the universal family around the nodal curve $\X_0$,
where only a nonseparating SCC $\gamma \subset R$ is pinched. In particular, the singular curves are parametrized by $z_1 = 0$.
Let
\[ U = \{ (z,\xi) \in \Delta \times \C : z_1^L = \xi \}.\]
Then $\rho: U \to \Delta$ is an $L$-cyclic cover,
branched along $z_1 = \xi = 0$. Let $U^*$ be the complement of $z_1 = \xi = 0$. The local monodromy for $\rho: \pi_1(\Delta^*) \to \Mod(S_g)$ is generated by $T_\gamma$,
and $\dprod{T_\gamma^L} = \rho^{-1}(\Gamma_g[\psi])$. It follows that the pullback family $\rho^*\X \to U$ gives a neighborhood (in $\overline{\CR_g[\psi]}$) of a point
$y \in Z[\psi]$, and the local monodromy around $y$ is generated by $T_\gamma^L$. Let $Z_y$ be the top stratum of the irreducible component of $\partial \CR_g[\psi]$
containing $y$. Then $U \cap \{z_1 = \xi =0 \} \subset Z_y$, so $Z_y$ is of codimension $1$ with local monodromy conjugate in
$\Gamma_g[\psi]$ to $T_\gamma^L$ for $\gamma \subset R$. It follows that $\overline{C}$ will intersect
$Z_y$, in particular $C$ is not compact.

This is enough to conclude that $\Prym[\psi](C)$ is rigid: Let $\X[L] \to \A_{g-1}[L]$ be the universal family of PPAVs
with level $L$ structure. Then, let $E[L] \to C$ be the pullback of $\X[L]$ under $\Prym \circ i:C \to \A_{g-1}[L]$.
Forgetting the level $L$ structure, gives a family $\rho: E \to C$ of PPAVs over $C$. As $\A_{g-1}[L] \to \A_{g-1}$
is a finite branched cover, it is enough to show that $E$ is rigid.

Since $i_*:\pi_1(C) \to \pi_1(\CR_g[\psi])$ is surjective and $\Prym_*(\Gamma_g[\psi]) = \Sp(2g - 2,\Z)[L]$, it follows
that the monodromy representation $\rho_*:\pi_1(C) \to \Sp(2g - 2,\Z)$ is irreducible.

Finally, there is a point $y' \in \overline{C} \cap Z_y$ so that the local monodromy of $E$ around $y$ is conjugate to
$\Prym_*(T_\gamma^L)$ for some $T_\gamma \in \Mod(R)[L]$ along a nonseparating SCC $\gamma$. As $T_\gamma$ maps to a transvection under $\Prym_*$ the local monodromy has infinite
order and the claim follows by applying~\cite[Thm 8.6]{saito1993classification}.

\begin{proof}[Proof of \Cref{theo:prym_rigid} via Rigid curves] Let $C$ be a rigid curve, then the homotopy $F_t:C \to \A_{g-1}[L]$
  between $F[\psi]$ and $\Prym[\psi]$, which is algebraic at each $t$ is \emph{constant} with respect to $t$, hence $\Prym[\psi]$ and $F[\psi]$ agree over $C$. As $\CR_g[\psi]$ is a quasiprojective variety and
$\h_{g-1}$ is a bounded domain, by the criterion of Borel-Narasimhan\cite[Thm 3.6]{Borel-Narasimhan}, it follows
that $F[\psi] = \Prym[\psi]$, hence also $\Prym = F$ and \Cref{theo:prym_rigid} is proven.
\end{proof}

\subsubsection{Finishing the Proof via VHS}\label{subsubsec:VHS}
Here we recount the approach suggested by R. Hain to Farb~\cite[Section 3]{farb2021global}. Let $C_L$ be a
complete curve in $\M_g[L]$ lying in the complement of the orbifold locus ($g \geq 3$), such that
the inclusion induces a surjection $\pi_1(C_L) \to \Mod(S_g)[L]$. Such a $C_L$ can be found
by slicing $\M_g[L]$ by hyperplanes on the satake compactification $\overline{\M}^S_g[L]$. Note that the forgetful map
$\phi:\CR_g[\psi] \to \M_g[L]$ is a branched cover, branched over the orbifold locus. Let $C_\psi$ be a
connected component of $\phi^{-1}(C_L)$, then $C_\psi$ is a complete curve, and the inclusion $i: C_\psi \to \CR_g[\psi]$
induces a surjection $i_*:\pi_1(C_\psi) \to \Gamma_g[\psi]$.
By the same arguments as in Step 4, $F[\psi] \circ i: C_\psi \to \A_{g-1}[L]$ is algebraic. Let $\VV \to \A_{g-1}[L]$
be the standard polarized variation of Hodge structures (PVHS) over $\A_{g-1}[L]$, so that over $A \in \A_{g-1}[L]$,
$\VV_A = H^1(A,\Z)$. Let $\VV_F$ and $\VV_{\Prym}$ be the pullbacks of $\VV$ over $C_\psi$, under the maps
$F[\psi] \circ i$ and $\Prym[\psi] \circ i$. Both monodromies are the same and are irreducible,
then it follows that (Theorem of the Fixed part and a Theorem of Schmid~\cite[Theorem 7.24]{schmid}),
\[ \VV_F \cong \VV_{\Prym}\]
as PVHS. Thus, there exist an $A \in \Sp(2g-2, \Z)$ inducing an isomorphism $A_L:\A_{g-1}[L] \to \A_{g-1}[L]$ such
that $A_L \circ F[\psi]$ and $\Prym[\Psi]$ agree over $C_\psi$. Composing with $A_L$ conjugates
\footnote{We can pick $A$ so that it makes the lifts of $A_L \circ F[\psi]$ and $\Prym[\psi]$ to the universal covers agree.}
the map
\[ (F[\psi] \circ i)_*:\pi_1(C_\psi) \to \Sp(2g-2,\Z)[L] \] by $A$. But $F[\psi]_* = \Prym[\psi]_*$, so A commutes with $\Sp(2g-2,\Z)[L]$
and thus $A = \{\pm \Id\}$. Consequentely $F[\psi] = \Prym[\psi]$ over $C_\psi$ and we are done by
citing Borel-Narasimhan~\cite{Borel-Narasimhan}.


\section{Appendix}
Let $S_g$ be a closed surface of genus $g \geq 1$, and $[\beta] \in H_1(S,\Z/2\Z)^*$. Then, there is a (unique up
to isomorphism) double cover \[ p:S_{2g-1} \to S_g, \]
with deck transform $\sigma$, and monodromy given by intersection with $[\beta]$.
In this section we provide a short proof of the following, which is weaker than the result of
Sato~\cite[Theorem 0.2]{Sato}, but suffices for our applications to holomorphic rigidity.
\begin{theo}
  Let $g \geq 4$, then the abelianization of $\Mbeta$,
  which we denote $\Mbeta^{\Ab}$,  is nontrivial cyclic of order at most $4$ and is $\Z/2\Z$ for $g$ even.
\end{theo}
\begin{proof}
  Let $b$ be a nonseparating SCC representing $[\beta]$, and $a$ be a SCC intersecting $b$ once transversely. Let $R$ be the
  complement of an open annulus neighborhood of $b$. Then $R \cong S_{g-1}^2$ and there is an inclusion $j: \Mod(R) \to \Mbeta$,
  with kernel generated by $T_{b'}T_{b''}^{-1}$ corresponding to twists along the boundary components of $R$.
  As $\Mod(R)^{\Ab} = 0$, it follows that $[T_c] = 0 \in \Mbeta^{\Ab}$ for any SCC $c$ disjoint from $b$.
  In particular, by \Cref{lemma:gen_dehn_pi1(Rg)}, $\Mbeta^{\Ab} = \dprod{[T_a^2]}$.
  \begin{figure}[H]
    \begin{center}
        \def \svgwidth{.7\textwidth}
\begingroup%
  \makeatletter%
  \providecommand\color[2][]{%
    \errmessage{(Inkscape) Color is used for the text in Inkscape, but the package 'color.sty' is not loaded}%
    \renewcommand\color[2][]{}%
  }%
  \providecommand\transparent[1]{%
    \errmessage{(Inkscape) Transparency is used (non-zero) for the text in Inkscape, but the package 'transparent.sty' is not loaded}%
    \renewcommand\transparent[1]{}%
  }%
  \providecommand\rotatebox[2]{#2}%
  \newcommand*\fsize{\dimexpr\f@size pt\relax}%
  \newcommand*\lineheight[1]{\fontsize{\fsize}{#1\fsize}\selectfont}%
  \ifx\svgwidth\undefined%
    \setlength{\unitlength}{281.84464582bp}%
    \ifx\svgscale\undefined%
      \relax%
    \else%
      \setlength{\unitlength}{\unitlength * \real{\svgscale}}%
    \fi%
  \else%
    \setlength{\unitlength}{\svgwidth}%
  \fi%
  \global\let\svgwidth\undefined%
  \global\let\svgscale\undefined%
  \makeatother%
  \begin{picture}(1,0.21702929)%
    \lineheight{1}%
    \setlength\tabcolsep{0pt}%
    \put(0,0){\includegraphics[width=\unitlength,page=1]{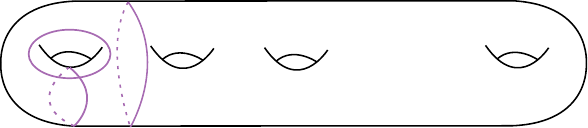}}%
    \put(0.04239652,0.05485623){\makebox(0,0)[lt]{\lineheight{1.25}\smash{\begin{tabular}[t]{l}$b$\end{tabular}}}}%
    \put(0.09846853,0.18505716){\makebox(0,0)[lt]{\lineheight{1.25}\smash{\begin{tabular}[t]{l}$a$\end{tabular}}}}%
    \put(0.26003169,0.16414902){\makebox(0,0)[lt]{\lineheight{1.25}\smash{\begin{tabular}[t]{l}$d$\end{tabular}}}}%
  \end{picture}%
\endgroup%

  \end{center}
  \caption{2-chain relation.}\label{fig:2-chain}
\end{figure}
  By the $2$-chain relation $(T_a^2T_b)^4 = T_d$. As $d$ is disjoint from $b$, $[T_a^2]^4 =0$.
  Similarly, using the $k$-relation depicted in \Cref{fig:chain_conj}
  \[ (T_a^2T_{c_1}T_{c_2})^3 = (T_{a'}^2T_{c_3}\ldots T_{c_{2g-2}})^{2g-3} .\]
  Thus, for $g$ even $[T_a^2]^2 = 0$.

  The nontriviality of $\Mbeta^{\Ab}$ follows from.
  \begin{lemma}
     Let $p \in \N$ and $\Lambda_g[p] = \{ A \in \Sp(2g;\Z); Ae_1 = e_1 + p a_1 \mbox{ , } a_1 \in \Sp(2g,\Z)\}$ and denote by $\land$ the symplectic
  pairing. The map $\varphi: \Lambda_g[p] \to
  \Z_p$ defined by
  \[ A \mapsto \frac{1}{p} (Ae_1 \land e_1) \mod p \]
  is a surjective homomorphism. In particular $H_1(\Lambda_g[p];\Z)$ is of order at least $p$.
\end{lemma}

\begin{proof}
  Let $A,B \in \Lambda_g[p]$, then:
  \[ (AB e_1) = A(p b_1 + e_1) = pAb_1 + pa_1 + e_1 \]
  Now as $A$ preserves the symplectic pairing $\land$, we have:
  \[ Ab_1 \land Ae_1 = Ab_1 \land (pa_1 + e_1) = b_1 \land e_1 \]
  And so we get $Ab_1 \land e_1 = b_1 \land e_1 -p Ab_1 \land a_1$.
  Hence:
  \[ (AB)e_1 \land e_1 = p ( a_1 \land e_1 + b_1 \land e_1 - p Ab_1 \land a_1) \]
  and so $\varphi$ is a group homomorphism. To see that it is surjective just take the $p$-th powers of the transvection
  given by $\begin{pmatrix} 1 & 0 \\ \pm1 & 1 \end{pmatrix}$.

\end{proof}
\end{proof}
The results of \Cref{subsec:rel-msigma} imply the following.
\begin{cor} Let $g\geq 4$, then $\Msigma^{\Ab}$ is cyclic of order at most $4$. Furthermore, it is generated
  by the class of $T_{\tilde{a}}$, where $a$ is a nonseparating SCC with $\hat{i}_2([a],[\beta]) =1$.
  For $g$ even, $[\sigma] = [T_{\tilde{a}}]^2$ and $[\sigma] = 0$ for $g$ odd.
\end{cor}

\begin{remark}Sato~\cite[Theorem 0.2]{Sato} has shown that $\Mbeta^{\Ab} = \Z/4\Z$ for $g$ odd, and $\Msigma^{\Ab} = \Z/4\Z$.
\end{remark}


\printbibliography
\end{document}
